\documentclass[11pt]{article}
\usepackage{amsfonts}
\usepackage{amssymb}
\usepackage{amsmath}
\usepackage{theorem}

\usepackage{hyperref}

\usepackage[normalem]{ulem}
\usepackage{tikz}
\usepackage{yfonts}
\usepackage{comment}

\usetikzlibrary{matrix}
\usetikzlibrary{arrows}

\numberwithin{equation}{section}

\input amssym.def

\makeatletter
\newcommand\qedsymbol{\hbox{$\Box$}}
\newcommand\qed{\relax\ifmmode\Box\else
  {\unskip\nobreak\hfil\penalty50\hskip1em\null\nobreak\hfil\qedsymbol
  \parfillskip=\z@\finalhyphendemerits=0\endgraf}\fi}

\newcommand\subqedsymbol{\hbox{$\triangledown$}}
\newcommand\subqed{\relax\ifmmode\triangledown\else
  {\unskip\nobreak\hfil\penalty50\hskip1em\null\nobreak\hfil\subqedsymbol
  \parfillskip=\z@\finalhyphendemerits=0\endgraf}\fi}
\makeatother

\newenvironment{proof}[1][{}]{\par\noindent Proof{#1}. }{\qed}

\newcommand{\hotimes}{{\,\hat{\otimes}\,}}

\newcommand{\Cbu}{C^{\bullet}}

\renewcommand{\deg}{{\mathrm{deg}}}

\newcommand{\Hom}{\mathrm{Hom}}

\newcommand{\tw}{\mathrm{tw}}
\newcommand{\TL}{\mathsf{TL}}
\newcommand{\mMC}{\mathbf{MC}}

\newcommand{\Hoch}{{\mathrm{Hoch}}}

\newcommand{\Sh}{{\mathrm{Sh}}}

\newcommand{\Shift}{{\mathrm{Shift}}}

\newcommand{\Der}{{\mathrm {D e r}}}

\newcommand{\bul}{{\bullet}}
\newcommand{\al}{{\alpha}}
\newcommand{\la}{{\lambda}}

\newcommand{\mm}{{\mathfrak{m}}}

\newcommand{\mG}{{\mathfrak{G}}}

\newcommand{\bs}{{\mathbf{s}}}

\newcommand{\bsi}{{\mathbf{s}^{-1}\,}}
\newcommand{\bsis}{{\mathbf{s}^{-2}\,}}

\newcommand{\cF}{\mathcal{F}}
\newcommand{\cG}{\mathcal{G}}
\newcommand{\cL}{\mathcal{L}}

\newcommand{\cO}{\mathcal{O}}

\newcommand{\cU}{\mathcal{U}}

\newcommand{\bbk}{{\Bbbk}}

\newcommand{\bbC}{{\mathbb C}}
\newcommand{\bbR}{{\mathbb R}}
\newcommand{\bbZ}{{\mathbb Z}}

\newcommand{\te}{\theta}

\newcommand{\de}{{\delta}}
\newcommand{\D}{{\Delta}}

\newcommand{\om}{{\omega}}
\newcommand{\Om}{{\Omega}}

\newcommand{\si}{{\sigma}}

\newcommand{\ve}{{\varepsilon}}
\newcommand{\ka}{{\kappa}}

\newcommand{\Te}{{\Theta}}

\newcommand{\sm}{\mathsf{m}}

\newcommand{\sJ}{\mathsf{J}}

\newcommand{\sPD}{\mathsf{PD}}

\newcommand{\sPV}{\mathsf{PV}}

\newcommand{\Omb}{{\Omega^{\bullet}}}

\newcommand{\wave}{\,\tilde{~}\,}

\newcommand{\ti}[1]{{\tilde{#1}}}

\newcommand{\und}[1]{{\underline{#1}}}

\newcommand{\del}{\partial}
\newcommand{\MC}{\mathrm{MC}}

\newtheorem{thm}{Theorem}[section]

\newtheorem{cor}[thm]{Corollary}
\newtheorem{conj}[thm]{Conjecture}
\newtheorem{prop}[thm]{Proposition}
\newtheorem{claim}[thm]{Claim}
\newtheorem{cond}[thm]{Condition}

\theorembodyfont{\rm}

\newtheorem{remark}[thm]{Remark}

\title{Towards deformation quantization over a $\bbZ$-graded base}

\author{E. Altinay-Ozaslan and V. Dolgushev}

\date{}

\begin{document}

\maketitle

\begin{flushright}
{\small\it To Murray Gerstenhaber, Jim Stasheff and Dennis Sullivan on the occasion of their jubilees}
\end{flushright}

~\\

\begin{abstract} 
The goal of this note is to describe a class of formal deformations of a symplectic 
manifold $M$ in the case when the base ring of the deformation problem 
involves parameters of non-positive degrees. The interesting feature of 
such deformations is that these are deformations ``in $A_{\infty}$-direction'' 
and, in general,  their description involves all cohomology classes of $M$ 
of degrees $\ge 2$.  
\end{abstract}

~

\section{Introduction}

Let $M$ be a real manifold $M$, $\cO(M)$ be the algebra of smooth complex-valued functions on $M$, 
and $\ve, \ve_1, \dots, \ve_g$ be formal variables of degrees 
\begin{equation}
\label{degrees}
\deg(\ve)=0\,, \quad \deg(\ve_1) = d_1, \quad  \deg(\ve_2) = d_2,  ~\dots, ~  \deg(\ve_g) = d_g\,,
\end{equation}
where $d_1, d_2, \dots, d_g$ are non-positive integers. 

In this paper, we investigate the problem of deformation quantization \cite{BFFLS}, \cite{BFFLS1}, 
\cite{Berezin}, \cite{BCG}, \cite{Deligne}, \cite{Fedosov}, \cite{K} of $M$ in the setting when we have several formal deformation 
parameters $\ve, \ve_1, \dots, \ve_g$ and some of the (non-positive) integers $d_1, d_2, \dots, d_g$ are 
actually non-zero. 

A formal deformation of $\cO(M)$, in this setting, is a (non-curved) 
$\bbC[[\ve, \ve_1, \dots, \ve_g]]$-linear $A_{\infty}$-structure on
$\cO(M)[[\ve, \ve_1, \dots, \ve_g]]$ with the multiplications $\{\sm_n\}_{n \ge 2}$ of the form
\begin{equation}
\label{mult-n}
\sm_n (a_1, \dots, a_n)  =
\begin{cases}
\displaystyle  a_1 a_2   + \sum_{k_0 d_0 + \dots + k_g d_g = 0 } 
\ve^{k_0} \ve_1^{k_1} \dots \ve_g^{k_g}
\mu_{k_0, k_1, \dots, k_g}  (a_1, a_2) \qquad {\rm if} ~~ n = 2, \\[0.8cm]
\displaystyle  \sum_{k_0 d_0 + \dots + k_g d_g = 2-n} 
\ve^{k_0} \ve_1^{k_1} \dots \ve_g^{k_g}
\mu_{k_0, k_1, \dots, k_g}  (a_1, \dots, a_n)  \qquad {\rm if} ~~ n > 2,
\end{cases} 
\end{equation}
where each $\mu_{k_0, k_1, \dots, k_g}$ is a polydifferential operator on $M$ (with complex coefficients)
acting on $2 -k_0 d_0 - \dots - k_g d_g $ arguments. Moreover,   $\mu_{k_0, k_1, \dots, k_g} \equiv 0$ if 
at least one $k_i < 0$ or $k_0 + k_1 + \dots + k_g  = 0$.

Such $A_{\infty}$-structures are in bijection with Maurer-Cartan (MC) elements of the dg Lie algebra 
\begin{equation}
\label{cL-intro}
(\ve, \ve_1, \dots, \ve_g)\, \sPD^{\bul}(M)[[\ve, \ve_1, \dots, \ve_g]],
\end{equation}
where $\sPD^{\bul}(M)$ denotes the algebra of polydifferential operators\footnote{See Section \ref{sec:sPD-sPV}.} on $M$. 

Let us denote by $\mG$ the group which is obtained by exponentiating the Lie algebra of degree zero elements of 
\eqref{cL-intro} and recall \cite{BDW}, \cite{Ezra-Deligne} that this group acts in a natural way on the set of MC elements of 
\eqref{cL-intro}. 

By analogy with $1$-parameter (``ungraded'') formal deformations, we declare that two such formal deformations are 
equivalent if the corresponding MC elements of \eqref{cL-intro} belong to the same orbit of the action of $\mG$. 

Let us observe that, for every MC element $\mu$ of \eqref{cL-intro}, the coset of $\mu$ in 
\begin{equation}
\label{quotient-intro}
(\ve, \ve_1, \dots, \ve_g)\, \sPD^{\bul}(M)[[\ve, \ve_1, \dots, \ve_g]] ~\big/~ (\ve, \ve_1, \dots, \ve_g)^2\, \sPD^{\bul}(M)[[\ve, \ve_1, \dots, \ve_g]]
\end{equation}
is closed with respect to the Hochschild differential $\del^{\Hoch}$. Moreover, if two MC elements $\mu_1$ and $\mu_2$ lie on 
the same orbit of $\mG$, then the corresponding cosets in \eqref{quotient-intro} are $\del^{\Hoch}$-cohomologous.  
By analogy with ``ungraded'' formal deformations, we call the $\del^{\Hoch}$-cohomology class of the coset of $\mu$ in
\eqref{quotient-intro}  the Kodaira-Spencer class of $\mu$. 

Let us recall that the cohomology space of $\sPD^{\bul}(M)$ (with respect to $\del^{\Hoch}$) is isomorphic to the space 
$\sPV^{\bul}(M)$ of polyvector fields
on $M$. So the Kodaira-Spencer class of every MC element of \eqref{cL-intro} can be identified 
with a degree $1$ vector in the graded space
$$
\ve \sPV^{\bul}(M) \oplus \ve_1 \sPV^{\bul}(M) \oplus \dots \oplus  \ve_g \sPV^{\bul}(M).
$$ 

Let us now assume that $M$ has a symplectic structure\footnote{In particular, it means that the dimension of $M$ is even.} 
$\om$ and denote by $\al \in  \sPV^{1}(M)$ the Poisson structure corresponding to $\om$.

In this paper, we consider formal deformations \eqref{mult-n} of $\cO(M)$ which satisfy these two conditions:  

\begin{enumerate}

\item the Kodaira-Spencer class of this deformation is $\ve\, \al$ and

\item 
$$
\sm_n  \big|_{\ve = 0} ~ = ~
\begin{cases}
\displaystyle  a_1 a_2   \qquad {\rm if} ~~ n = 2, \\[0.3cm]
 0 \qquad {\rm if} ~~ n > 2.
\end{cases} 
$$

\end{enumerate}

We denote by $\TL$ the set of equivalence classes of formal deformations \eqref{mult-n} satisfying the above 
conditions and call $\TL$ the {\it topological locus} of the triple $(M, \om, \{\ve, \ve_1, \dots, \ve_g \})$. 
Using Kontsevich's formality \cite{K} and a construction inspired by paper \cite{ShT} 
due to Sharygin and Talalaev, we give a description of the topological locus $\TL$ in terms of 
the singular cohomology of $M$. More precisely,  
\begin{thm}
\label{thm:intro}
For every symplectic manifold $(M,\om)$, the equivalence classes of formal deformations 
\eqref{mult-n} of $\cO(M)$ satisfying the above conditions are 
in bijections with degree $2$ vectors of the graded vector space
$$
\bigoplus_{q \geq 0}  \, (\ve, \ve_1, \dots, \ve_g) \, H^q(M, \bbC)[[\ve, \ve_1, \dots, \ve_g]],
$$
where $H^{\bul}(M, \bbC)$ is the singular cohomology of $M$ with coefficients in $\bbC$ and 
every vector of $H^{q}(M, \bbC)$ carries degree $q$. 
\end{thm}
\begin{remark}
\label{rem:BCG}
In the ``ungraded'' case (i.e. $g = 0$), this result reproduces the classical theorem 
\cite{BCG}, \cite{Deligne} of Bertelson, Deligne, Cahen, and Gutt on the description of the equivalence 
classes of star products on a symplectic manifold. In this respect, Theorem \ref{thm:intro} 
may be viewed as a generalization of this classification theorem
to the case of a $\bbZ$-graded base.
\end{remark}
\begin{remark}
\label{rem:Z-graded}
Deformations over a $\bbZ$-graded (and even differential graded) base were considered in the literature. 
See, for example, paper \cite{BK} by Barannikov and Kontsevich or J. Lurie's ICM address \cite{LurieDAGX}
in which even more sophisticated examples of bases for deformation problems were considered.  
We should also mention that deformations over a differential graded base naturally show up 
in the construction of rational homotopy models for classifying spaces of fibrations. For more 
details, we refer the reader to paper \cite{Lazar} by Lazarev.  
\end{remark}

\paragraph*{Organization of the paper.} The remainder of the introduction is devoted to the notational 
conventions and preliminaries. In this part, we give a brief reminder of the Deligne-Getzler-Hinich(DGH) 
groupoid(s) and fix our conventions related to the sheaf of polydifferential operators 
and polyvector fields. 

Section \ref{sec:gen-story} starts with a brief reminder of $1$-parameter 
formal deformations of an associative algebra (over $\bbC$). Then we propose a natural generalization 
of this story to the case when the base ring of the deformation problem is the completion 
of the free graded\footnote{$\bbC[\ve, \ve_1, \dots, \ve_g]$ should not be confused with the polynomial 
algebra because $\ve_i \ve_j = (-1)^{d_i d_j} \ve_j \ve_i$.} 
commutative algebra $\bbC[\ve, \ve_1, \dots, \ve_g]$. Next, we consider the case when $A$ is the 
algebra of functions $\cO(M)$ on a smooth real manifold $M$. Finally, we assume that $M$ has 
a symplectic structure and formulate the main result of this paper (see Theorem \ref{thm:main}). 

The proof of Theorem \ref{thm:main} is given in Section \ref{sec:proof} and it is based on 
two auxiliary constructions which are presented in Sections \ref{sec:PV-PD} and \ref{sec:PV}, respectively. 
In Section \ref{sec:conj}, we propose two conjectures related to Theorem \ref{thm:main}. 
Finally, Appendix \ref{app:Pi} is devoted to the proof of a technical proposition.

~\\

\subsection{Notational conventions and preliminaries}
We assume that the ground field is the field of complex numbers $\bbC$ 
and set $\otimes : = \otimes_{\bbC}$, $\Hom : = \Hom_{\bbC}$. 
For a cochain complex $V$ we denote 
by $\bs V$ (resp. by $\bs^{-1} V$) the suspension (resp. the 
desuspension) of $V$\,. In other words, 
$$
\big(\bs V\big)^{\bul} = V^{\bul-1}\,,  \qquad
\big(\bs^{-1} V\big)^{\bul} = V^{\bul+1}\,. 
$$

The notation $S_{n}$ is reserved for the symmetric group 
on $n$ letters and  $\Sh_{p_1, \dots, p_k}$ denotes 
the subset of $(p_1, \dots, p_k)$-shuffles 
in $S_n$, i.e.  $\Sh_{p_1, \dots, p_k}$ consists of 
elements $\si \in S_n$, $n= p_1 +p_2 + \dots + p_k$ such that 
$$
\begin{array}{c}
\si(1) <  \si(2) < \dots < \si(p_1),  \\[0.3cm]
\si(p_1+1) <  \si(p_1+2) < \dots < \si(p_1+p_2), \\[0.3cm]
\dots   \\[0.3cm]
\si(n-p_k+1) <  \si(n-p_k+2) < \dots < \si(n)\,.
\end{array}
$$

For a groupoid $\cG$, $\pi_0(\cG)$ denotes the set of isomorphism classes of 
objects of $\cG$. For a graded vector space (or a cochain complex) $V$
the notation $S(V)$ (resp. $\und{S}(V)$) is reserved for the
underlying vector space of the
symmetric algebra (resp. the truncated symmetric algebra) of $V$: 
$$
S(V) = \bbC \oplus V \oplus S^2(V) \oplus S^3(V) \oplus \dots\,, 
$$ 
$$
\und{S}(V) =  V \oplus S^2(V) \oplus S^3(V) \oplus \dots\,,
$$
where 
$$
S^n(V) = \big( V^{\otimes\, n} \big)_{S_n}
$$
is the space of coinvariants with respect to the obvious action of $S_n$.

Recall that $\und{S}(V)$ is the vector space of the cofree cocommutative coalgebra 
(without counit) cogenerated by $V$. The comultiplication on $\und{S}(V)$ is given 
by the formula:
$$
\D(v_1, v_2, \dots, v_n) : = \sum_{p = 1}^{n-1} 
\sum_{\si \in \Sh_{p, n-p} } 
(-1)^{\ve(\si; v_1, \dots, v_n)} (v_{\si(1)}, \dots, v_{\si(p)}) \otimes 
(v_{\si(p+1)}, \dots, v_{\si(n)})\,,
$$
where $(-1)^{\ve(\si; v_1, \dots, v_n)}$ is the Koszul sign factor
\begin{equation}
\label{ve-si-vvv}
(-1)^{\ve(\si; v_1, \dots, v_n)} := \prod_{(i < j)} (-1)^{|v_i | |v_j|}
\end{equation}
and the product in \eqref{ve-si-vvv} is taken over all inversions $(i < j)$ of $\si \in S_n$.

For an associative algebra $A$, we denote by 
$\Cbu(A)$ the Hochschild cochain complex of $A$ with coefficients in $A$
and with the shifted grading: 
\begin{equation}
\label{Hoch-A}
\Cbu(A) = \bigoplus_{k \ge -1} C^k(A), \qquad 
C^k(A) : = \Hom(A^{\otimes \,(k+1)}, A).
\end{equation}
We denote by $[~,~]_G$ the Gerstenhaber bracket \cite{Ger} on $\Cbu(A)$: 
\begin{equation}
\label{Ger-brack}
[P_1, P_2]_{G} (a_0, \dots, a_{k_1+ k_2}) : = 
\end{equation}
$$
\sum_{i=0} (-1)^{i k_2} P_1 (a_0, \dots, a_{i-1}, P_2(a_i, \dots, a_{i+k_2}), a_{i+k_2+1}, \dots, a_{k_1+k_2})
- (-1)^{k_1 k_2} (1 \leftrightarrow 2),
$$
where $P_j  \in C^{k_j}(A)$.

It is convenient to think of the multiplication on $A$ as 
the Hochschild cochain $m_A \in C^1(A)$ and define the Hochschild 
differential $\del^{\Hoch}$ on $\Cbu(A)$ as
\begin{equation}
\label{del-Hoch}
\del^{\Hoch} : = [m_A , ~]_G \,. 
\end{equation}

We reserve the notation $HH^{\bul}(A)$ for the Hochschild cohomology  
of $A$ with coefficients in $A$, i.e. 
$$
HH^{\bul}(A) : = H^{\bul}\big(\Cbu(A) \big).   
$$
For example, if $A$ is the polynomial algebra $\bbC[x^1, \dots, x^m]$ is $m$ variables
then \cite[Section 3.2]{Loday}
\begin{equation}
\label{HH-poly-ls}
HH^{\bul}(A) \cong S_A\big( \bs \Der (A) \big),
\end{equation}
where $S_A(E)$ denotes the symmetric algebra of an $A$-module $E$ and 
$\Der(A)$ is the $A$-module of $\bbC$-linear derivations. 

If  $\ve, \ve_1, \dots, \ve_g$ are variables of degrees 
$$
\deg(\ve)=d_0\,, \quad \deg(\ve_1) = d_1, \quad  \deg(\ve_2) = d_2, ~ \dots,  ~~ \deg(\ve_g) = d_g\,,
$$
then the notation 
\begin{equation}
\label{C-veveve}
\bbC[\ve, \ve_1, \dots, \ve_g]
\end{equation}
is reserved for the free \und{graded} commutative algebra over $\bbC$
generated by $\ve, \ve_1, \dots, \ve_g$. Furthermore, 
$$
\bbC[[\ve, \ve_1, \dots, \ve_g]]
$$
denotes the completion of \eqref{C-veveve} with respect to the ideal 
$\mm = (\ve, \ve_1, \dots, \ve_g) \subset \bbC[\ve, \ve_1, \dots, \ve_g]$. 
For example, if $\ve_1$ carries an odd degree then $\ve_1^2 = 0$ in \eqref{C-veveve} 
and in its completion.  

For every dg Lie algebra $(L, \del, [~,~])$, the cofree cocommutative coalgebra 
\begin{equation}
\label{CE-L}
\und{S}(\bsi L)
\end{equation}
is equipped with a degree $1$ coderivation $Q$ which satisfies $Q^2 = 0$.
The composition\footnote{It is not hard to see that every coderivation $Q$ of a cofree cocommutative 
coalgebra \eqref{CE-L} is uniquely determined by the composition $p_{\bsi L} \circ Q : \und{S}(\bsi L) \to \bsi L$. } 
$p_{\bsi L} \circ Q$ of $Q$ with the canonical projection $p_{\bsi L}: \und{S}(\bsi L) \to \bsi L$
is expressed in terms of $\del$ and $[~,~]$ as follows: 
\begin{equation}
\label{CE-Q}
p_{\bsi L} \circ Q (\bsi v_1 \dots \bsi v_n) : = 
\begin{cases}
 \bsi (\del v_1) \qquad {\rm if} ~~ n =1, \\[0.18cm]
(-1)^{|v_1| - 1} \bsi [v_1, v_2] \qquad {\rm if} ~~ n=2, \\[0.18cm]
 0 \qquad {\rm if} ~~ n \ge 3.
\end{cases}
\end{equation}
The assignment $(L, \del, [~,~]) \mapsto \big( \und{S}(\bsi L), Q, \D \big)$ is often called 
the Chevalley-Eilenberg construction.  

Let us recall that an $L_{\infty}$-morphism $F$ from a dg Lie algebra 
$(L, \del, [~,~])$ to a dg Lie algebra  $(\ti{L}, \ti{\del}, [~,~]\wave)$ is a homomorphism 
of the corresponding dg cocommutative coalgebras 
\begin{equation}
\label{L-infty-def}
F : (\und{S}(\bsi L), Q) \to  (\und{S}(\bsi \ti{L}), \ti{Q}).
\end{equation}

It is not hard to see that every coalgebra homomorphism \eqref{L-infty-def} is 
uniquely determined by its composition 
$$
p_{\bsi \ti{L}} \circ F \, : \, \und{S}(\bsi L) ~ \to ~ \bsi \ti{L}
$$
with the canonical projection $p_{\bsi \ti{L}} : \und{S}(\bsi \ti{L}) \to  \bsi \ti{L}$. 
In this paper, we denote by $F_n$ the restriction of $p_{\bsi \ti{L}} \circ F$ onto $S^n (\bsi L)$: 
$$
F_n : = p_{\bsi \ti{L}} \circ F \big|_{ \, S^n (\bsi L) \, } : S^n (\bsi L) \to  \bsi \ti{L}
$$
and call $F_1$ the {\it linear term} of the $L_{\infty}$-morphism $F$.  Recall that, for every $L_{\infty}$-morphism $F$, 
its linear term $F_1$ is a chain map $(L, \del) \to (\ti{L}, \ti{\del})$.

Most of the differential graded (dg) Lie algebras $(L, \del, [~,~])$, we consider, are equipped with
a complete descending filtration 
\begin{equation}
\label{filtr-L}
\dots \supset \cF_0 L  \supset \cF_1 L \supset \cF_2 L \supset  \cF_3 L \supset \dots 
\qquad L \cong \lim_k L \big/ \cF_k L.
\end{equation}
Furthermore, we tacitly assume that our $L_{\infty}$-morphisms are compatible with the filtrations in the 
following sense 
\begin{equation}
\label{compat-w-filtr}
F_n  \big( \bsi \cF_{k_1} L \otimes \dots \otimes \bsi \cF_{k_n} L  \big)  ~\subset~  \bsi \cF_{k_1+ \dots + k_n} \ti{L}.
\end{equation}

For a smooth real manifold $M$, the notation $\cO_M$ (resp. $\cO(M)$) is reserved for the 
sheaf (resp. the algebra) of smooth complex valued functions on $M$. 
The symbols $x^1, x^2, \dots, x^m$ are often reserved for coordinates on an open subset $U \subset M$. 
$TM$ (resp. $T^*M$) denotes the tangent (resp. cotangent) bundle of $M$. Moreover, 
$\{ d x^1,  d x^2, \dots, dx^m \}$ and $\{ \te_1,  \te_2, \dots, \te_m \}$ will be the standard local 
frames\footnote{In other words, $\te_i : = \del/\del x^i\,.$} 
for $T^* M \big|_{U}$ and $TM \big|_U$ corresponding to coordinates $x^1, x^2, \dots, x^m$, respectively. 
In particular, the graded commutative algebra $\Omb(U)$ (resp. $\wedge^{\bul} TM (U)$) of exterior forms (resp. polyvector fields)
on $U$ will be tacitly identified with\footnote{Recall that, since the symbols $d x^i$ and $\te_i$ carry degree $1$, we have 
$dx^i dx^j = - dx^j dx^i$ and $\te_i \te_j = - \te_j \te_i$.} 
$\cO_M(U)[d x^1, d x^2, \dots, d x^m]$ (resp. $\cO_M(U)[\te_1, \te_2, \dots, \te_m]$). 
Occasionally, we will use the (left) ``partial derivative'' 
$$
\frac{\del}{\del\, dx^i} : \Omb(U) \to \Om^{\bul-1}(U)
$$  
with respect to the degree $1$ symbol $d x^i$. This operation is defined by the formula
$$
\frac{\del}{\del\, dx^i} \, \eta_{i_1 \dots i_k}(x) d x^{i_1}  d x^{i_2} \dots d x^{i_k} ~: =~
k \eta_{i i_2 \dots i_k}(x) d x^{i_2}  d x^{i_3} \dots d x^{i_k}\,,
$$
where the summation over repeated indices is assumed. 
Equivalently, $\displaystyle \frac{\del \, \eta}{\del\, dx^i}$ is the contraction of 
an exterior form $\eta$ with the local vector field $\del/ \del x^i$.

\subsubsection{A reminder of the Deligne-Getzler-Hinich(DGH) groupoid(s)}
\label{sec:DGH}
Let $L$ be a dg Lie algebra equipped with a complete descending filtration \eqref{filtr-L}.
A Maurer-Cartan element of $L$ is a degree $1$ element $\mu \in \cF_1 L$ satisfying the equation 
\begin{equation}
\label{MC-eq}
\del \mu + \frac{1}{2}[\mu, \mu] = 0. 
\end{equation}
In this paper, $\MC(L)$ denotes the set of Maurer-Cartan elements of a dg Lie algebra $L$.

Let us recall \cite[Appendix B]{BDW}, \cite{Ezra-Deligne} that the formula 
\begin{equation}
\label{action}
e^{\xi} (\mu) : = e^{[\xi, ~]} \mu - \frac{e^{[\xi, ~]} - 1}{[\xi,~]} (\del\, \xi),
\qquad \xi \in \cF_1 L^0
\end{equation}
defines an action of the group
\begin{equation}
\label{exp-L-0}
\exp(\cF_1 L^0)
\end{equation}
on the set of MC elements of $L$. In \eqref{action}, the expressions  $e^{[\xi, ~]}$ and 
$$
 \frac{e^{[\xi, ~]} - 1}{[\xi,~]}
$$
are defined by the Taylor expansion of the functions $e^x$ and $(e^x-1)/x$, respectively, around 
the point $x = 0$. We denote by $\cG(L)$ the transformation groupoid of the action \eqref{exp-L-0}.


To a dg Lie algebra $L$ (equipped with a complete filtration \eqref{filtr-L}), we can also 
associate a very useful simplicial set $\mMC_{\bul}(L)$ \cite{Ezra-infty}, \cite{Hinich:1997} with
\begin{equation}
\label{int-mMC}
\mMC_n(L) : = \MC(L \hotimes \Om_n),
\end{equation}
where 
$$
L \hotimes \Om_n  : = \varprojlim_{k} \big( (L/\cF_{k}L) \otimes \Om_n \big)
$$
and $\Om_n$ is the de Rham-Sullivan algebra of polynomial 
differential forms on the geometric simplex $\Delta^n$ (with coefficients in $\bbC$).

As the graded commutative algebra (with $1$), 
$\Om_n$ is generated by $n+1$ symbols $t_0, t_1, \dots, t_n$ of degree $0$
and $n+1$ symbols $d t_0, d t_1, \dots, d t_n$ of degree $1$ subject to the relations
$$
t_0 + t_1 + \dots + t_n = 1, \qquad d t_0 + d t_1 + \dots + d t_n = 0. 
$$
Furthermore, the differential $d_t$ on $\Om_n$ is defined by the formulas
$$
d_t (t_i) : = d t_i, \qquad  d_t (d t_i) : = 0. 
$$
For example, $\Om_0 = \bbC$ and $\Om_1 \cong \bbC[t] \,\oplus\, \bbC[t] \, dt$. 

Due to \cite[Proposition 4.1]{enhanced}, the simplicial set $\mMC_{\bul}(L)$ is a Kan complex 
(a.k.a. an $\infty$-groupoid). Moreover, due to \cite[Lemma B.2]{GM}, two MC elements $\mu$ and $\ti{\mu}$
of $L$ (i.e. $0$-cells of $\mMC_{\bul}(L)$) are connected by a $1$-cell if and only if 
they belong to the same orbit of the action \eqref{action}. In other words, we have 
the identification:
\begin{equation}
\label{pi-0-same}
\pi_0 \big( \cG(L) \big) \cong \pi_0\big( \mMC_{\bul}(L) \big).
\end{equation}

\begin{remark}
\label{rem:infinity}
The Kan complex (a.k.a. a fibrant simplicial set) $\mMC_{\bul}(L)$ is called the 
{\it Deligne-Getzler-Hinich (DGH) $\infty$-groupoid.} In this paper, we mostly use 
the ``truncation'' of  $\mMC_{\bul}(L)$, i.e. the honest (transformation) groupoid $\cG(L)$. 
\end{remark}

Recall \cite[Section 2]{GM} that, for every $L_{\infty}$-morphism $F : L \leadsto \ti{L}$, the formula
$$
F_*(\mu) : = \sum_{n=1}^{\infty}\frac{1}{n!} F_n \big((\bsi \mu)^n \big) 
$$  
defines a map of sets 
\begin{equation}
\label{F-on-MC}
F_* : \MC(L) \to \MC(\ti{L}). 
\end{equation}
Furthermore, since $F$ naturally extends to an $L_{\infty}$-morphism 
$F : L \hotimes \Om_n \leadsto \ti{L} \hotimes \Om_n$ for every $n \ge 1$, the map $F_*$ 
naturally upgrades to the morphism of simplicial sets 
\begin{equation}
\label{F-on-mMC}
F_* : \mMC_{\bul}(L) \to \mMC_{\bul}(\ti{L})
\end{equation}
for which we use the same notation. Therefore $F_*$ gives us a map of sets\footnote{Here we use the identification 
\eqref{pi-0-same}.} 
\begin{equation}
\label{pi-0-F}
\pi_0 (F_*) : \pi_0 \big( \cG(L) \big) \to \pi_0 \big( \cG(\ti{L}) \big). 
\end{equation}

\begin{remark}
\label{rem:functor}
It is not hard to see that the assignments $L \mapsto  \mMC_{\bul}(L)$ and $F \mapsto F_*$ 
define a functor from the category of filtered dg Lie algebras to the category of simplicial sets.   
\end{remark}

\subsubsection{The sheaf $\sPD^{\bul}$ of polydifferential operators and the sheaf $\sPV^{\bul}$ of polyvector fields}
\label{sec:sPD-sPV}

Let $U \subset M$ be an open coordinate subset of a manifold $M$ 
with coordinates $x^1, x^2, \dots, x^m$.

For every $k \ge -1$, the space $\sPD^{k}(U)$ consists of $\bbC$-multilinear maps 
$$
P : \cO_M(U)^{\otimes k+1} \to \cO_M(U) 
$$ 
which can be written (in local coordinates) as finite sums
\begin{equation}
\label{P-local-form}
P ~ = ~  \sum_{\al_0, \al_1, \dots, \al_k} 
P^{\al_0, \al_1, \dots, \al_k} (x) \, \del_{x^{\al_0}} \otimes  \del_{x^{\al_1}}  \otimes \dots \otimes  \del_{x^{\al_k}} \,,
\end{equation}
where $\al_j$ are multi-indices, $P^{\al_0, \al_1, \dots, \al_k} (x) \in \cO_M(U)$
and, if $\al = (i_1, \dots, i_s)$, then
$$
\del_{x^{\al}} = \del_{x^{i_1}} \del_{x^{i_2}} \dots \del_{x^{i_s}}\,.
$$

For example, $\sPD^{-1} : = \cO_M$ and $\sPD^0$ is the sheaf of differential operators on $M$. 

We do consider polydifferential operators which do not necessarily annihilate constant functions. 
For example, the usual (commutative) multiplication $m_{\cO_M}$ can be viewed as the global section of $\sPD^{1}$.

It is easy to see that the Gerstenhaber bracket \eqref{Ger-brack} is defined on sections of the sheaf 
\begin{equation}
\label{sPD-bul}
\sPD^{\bul} : = \bigoplus_{k \ge -1} \sPD^{k}\,. 
\end{equation}
Thus $\sPD^{\bul}$ is a sheaf of graded Lie algebras. 

It is also easy to see that the formula 
\begin{equation}
\label{del-Hoch-sPD}
\del^{\Hoch} : = [m_{\cO_M}\,, ~] 
\end{equation}
defines a differential on $\sPD^{\bul}$ which is compatible with the Lie bracket $[~,~]_G$. 
So we view  $\sPD^{\bul}$ as a sheaf of dg Lie algebras.  

Let us denote by $\sPV^{k}$ the sheaf of local sections of $\wedge^{k+1} TM$ and set
$$
\sPV^{\bul} : = \bigoplus_{k \ge -1} \sPV^{k}\,.
$$ 
We call $\sPV^{\bul}$ the sheaf of polyvector fields on $M$. 

Since the graded commutative algebra $\wedge^{\bul} TM(U)$ is identified with
$$
\cO_M(U)[ \te_1,  \te_2, \dots, \te_m], 
$$
where $ \te_1,  \te_2, \dots, \te_m$ are degree $1$ symbols, every polyvector field $v$
of degree $k$ has the unique expansion 
$$
v = \sum v^{i_0 i_1 \dots i_k} (x)   \te_{i_0} \te_{i_1} \dots \te_{i_k}\,, \qquad  v^{i_0 i_1 \dots i_k} (x)  \in \cO_M(U),
$$
$$
v^{i_0 \dots i_{t} i_{t+1} \dots i_k} (x)  = - v^{i_0 \dots i_{t+1} i_{t} \dots i_k} (x).
$$
The functions $v^{i_0 i_1 \dots i_k} (x)$ are called components of $v \in \sPV^k(U)$.   

Recall \cite{Koszul} that $\sPV^{\bul}$ is a sheaf of Lie algebras. The Lie bracket $[~,~]_S$ 
(known as the Schouten bracket) is defined locally by the equations
\begin{equation}
\label{Schouten}
[x^i, x^j]_S = [\te_i, \te_j]_S = 0\,, 
\qquad 
[\te_i, x^j]_S  = - [x_j, \te_i]_S = \de_i^j
\end{equation}
and the Leibniz compatibility condition with the multiplication
$$
[v, v_1 v_2]_S = [v, v_1]_S\, v_2  + (-1)^{(|v_1| +1) |v|} v_1\, [v, v_2]_S\,.
$$

Let us also recall the every polyvector field $v \in \sPV^k(U)$ can be identified 
with the polydifferential operator in $\sPD^k(U)$ which acts as
$$
a_0 \otimes a_1 \otimes \dots \otimes a_k ~ \mapsto ~ 
\sum  v^{i_0 i_1 \dots i_k} (x) (\del_{x^{i_k}} a_0)  (\del_{x^{i_{k-1}}} a_1) \dots  (\del_{x^{i_0}} a_k),
$$
where $ v^{i_0 i_1 \dots i_k} (x)$ are components of $v$ and $x^1, \dots, x^m$ are 
coordinates on $U$. 

This embedding of sheaves 
\begin{equation}
\label{HKR}
\sPV^{\bul} \hookrightarrow  \sPD^{\bul}
\end{equation}
is often called the Hochschild-Kostant-Rosenberg (HKR) map \cite{HKR}. 

Clearly, every polyvector field is a $\del^{\Hoch}$-closed polydifferential operator.
So \eqref{HKR} is a chain map from the graded sheaf $\sPV^{\bul}$ with the zero differential 
to the graded sheaf  $\sPD^{\bul}$ with the Hochschild differential $\del^{\Hoch}$. 

We claim that 
\begin{prop}[M. Kontsevich, Section 4.6.1.1, \cite{K}]
\label{prop:HKR}
The embedding \eqref{HKR} gives us a quasi-isomorphism of cochain complexes
\begin{equation}
\label{HKR-M}
\big( \sPV^{\bul}(M), 0 \big) \stackrel{\sim}{\longrightarrow}  \big( \sPD^{\bul}(M), \del^{\Hoch} \big).
\end{equation}
In other words, $H^{k} \big( \sPD^{\bul}(M), \del^{\Hoch} \big)  \cong \sPV^{k}(M)$
for all $k$. 
\end{prop}
\begin{remark}
\label{rem:HKR}
A version of Proposition \ref{prop:HKR} in the setting of algebraic geometry 
is known as the Hochschild-Kostant-Rosenberg theorem \cite{HKR}.
\end{remark}

One has to be careful with the above identification of  $\sPV^{\bul}$ (resp. $\sPV^{\bul}(M)$)
with the corresponding subsheaf of $\sPD^{\bul}$ (resp. subspace of $\sPD^{\bul}(M)$) because 
the embedding \eqref{HKR} is compatible with the Lie brackets only {\it up to homotopy}. So in general, 
$$
[v,w]_S \neq [v, w]_{G}\,, \qquad v,w \in  \sPV^{\bul}(M).
$$

On the other hand, we have celebrated Kontsevich's formality theorem 
which states that 
\begin{thm}[M. Kontsevich, Section 7.3.5, \cite{K}]
\label{thm:K}
The exists a sequence of quasi-iso\-mor\-phisms of dg Lie algebras 
which connects $\big( \sPV^{\bul}(M), 0, [~,~]_S \big)$ with 
$\big( \sPD^{\bul}(M), \del^{\Hoch}, [~,~]_G \big)$.
\end{thm}
\begin{remark}
\label{rem:global}
Paper \cite{K} gives us an explicit sequence of quasi-isomorphisms of dg Lie algebras 
connecting  $\big( \sPV^{\bul}(\bbR^m), 0, [~,~]_S \big)$ with 
$\big( \sPD^{\bul}(\bbR^m), \del^{\Hoch}, [~,~]_G \big)$. For a detailed proof 
of Theorem \ref{thm:K} for an arbitrary smooth manifold, we refer the reader to 
\cite{CEFT}, \cite[Appendix A.3]{K-alg}. 
\end{remark}

~\\

%
%
\noindent
\textbf{Acknowledgements:} This paper may be considered as an answer 
to Yuri Berest's question: ``What if the (quantum) deformation parameter has 
a non-zero degree?'' So we would like to thank Yuri for posing this question.
We would also like to thank Brian Paljug and Xiang Tang for numerous discussions related to this paper. 
V.D. acknowledges NSF grants DMS-1161867 and DMS-1501001 for a partial support. 
E. A-O. was supported by the Republic of Turkey Ministry of National Education during 
her graduate studies at Temple University and she acknowledges this support. 
A part of this paper was written when V.D. was spending his sabbatical at the University of 
Pennsylvania. V.D. is grateful to the Math Department at the University of 
Pennsylvania for perfect working conditions and for the stimulating atmosphere. 

\section{Formal deformations of an associative algebra}
\label{sec:gen-story}

Let $A$ be an associative algebra over\footnote{The general deformation theory works for any 
ground field of characteristic zero.} $\bbC$. Let us recall that formal deformations of 
$A$ with the base ring $\bbC[[\ve]]$ are in bijection with MC elements 
\begin{equation}
\label{mu-series}
\mu = \sum_{k \ge 1} \ve^k \mu_k
\end{equation} 
of the dg Lie algebra 
\begin{equation}
\label{Cbu-A-veveve}
\ve \Cbu(A) [[\ve]]
\end{equation}
and two deformations are equivalent if the corresponding MC elements are isomorphic, i.e. 
lie on the same orbit of the action \eqref{action} (with $L = \ve \Cbu(A) [[\ve]]$).

Indeed, every MC element $\mu \in \ve C^1(A) [[\ve]]$ gives us an associative 
multiplication on $A[[\ve]]$: 
\begin{equation}
\label{mu-mult}
a \bullet_{\mu} b : = a b + \mu(a,b). 
\end{equation}
Furthermore, if 
$$
\ti{\mu} = e^{[\xi, ~]} \mu - \frac{e^{[\xi, ~]} - 1}{[\xi,~]} (\del^{\Hoch}\, \xi)
$$
for some $\xi \in  \ve C^0(A) [[\ve]]$ then the operator 
$$
T_{\xi} : A[[\ve]] \to A[[\ve]]\,, 
\qquad 
T_{\xi}(a) : = a + \sum_{k=1}^{\infty} \frac{1}{k!} \xi^k(a)
$$
intertwines the multiplications $\bullet_{\mu}$ and $\bullet_{\ti{\mu}}$: 
$$
T_{\xi}(a \bullet_{\mu} b) = T_{\xi}(a) \bullet_{\ti{\mu}} T_{\xi}(b)\,, \qquad \forall ~ a, b \in A[[\ve]].  
$$

Thus equivalence classes of $1$-parameter formal deformations of $A$ are 
in bijection with elements in 
$$
\pi_0 \big( \cG( \ve \Cbu(A) [[\ve]] )  \big).
$$

MC equation \eqref{MC-eq} implies that the first term $\mu_1$ of $\mu$ in \eqref{mu-series} is 
necessarily a degree $1$ cocycle in $\Cbu(A)$. The cohomology class $\ka$ of this cocycle
in $HH^1(A)$ depends only on the isomorphism class of the MC element\footnote{In other words, 
if $\mu$ is isomorphic to $\ti{\mu}$ then the corresponding cocycles in $C^1(A)$ are cohomologous.} $\mu$. 
This cohomology class is traditionally \cite{Ger}, \cite{KS} called the Kodaira-Spencer class of $\mu$. 

MC equation \eqref{MC-eq} also implies that the Kodaira-Spencer class $\ka$ of any 
formal deformation satisfies the ``integrability'' condition 
\begin{equation}
\label{ka-ka-zero}
[\ka, \ka] = 0,
\end{equation}
where $[~ , ~]$ is the induced Lie bracket on $HH^{\bul}(A)$.

%
%

By analogy with the above ``classical'' $1$-parameter case, we define formal 
deformations of $A$ with the base ring\footnote{In this paper, we assume that the 
formal variables $\ve, \ve_1, \dots, \ve_g$ have non-positive degrees.} 
$$
\bbk : = \bbC [[\ve, \ve_1, \dots, \ve_g]]
$$
as MC elements 
\begin{equation}
\label{mu-ve-graded}
\mu =  \sum_{k_0 + k_1 + \dots + k_g \ge 1} 
\ve^{k_0} \ve_1^{k_1} \dots \ve_g^{k_g}
\mu_{k_0, k_1, \dots, k_g}\,, \qquad \mu_{k_0, k_1, \dots, k_g} \in C^{1 - (k_0 d_0 + \dots + k_g d_g)}(A) 
\end{equation}
of the dg Lie algebra 
\begin{equation}
\label{Cbu-A-ve-graded}
\mm \Cbu(A) [[\ve,  \ve_1, \dots, \ve_g]],
\end{equation} 
where  $\mm$ is the maximal ideal 
$$
\mm = (\ve, \ve_1, \dots, \ve_g) \subset  \bbC [\ve, \ve_1, \dots, \ve_g]. 
$$
Furthermore, we declare that two such deformations are equivalent if 
the corresponding MC elements are isomorphic. 

An interesting feature of such deformations is that, if at least one
formal parameter carries a non-zero degree, then the resulting 
MC element $\mu$ corresponds to (a $\bbC[[\ve,  \ve_1, \dots, \ve_g]]$-linear) 
$A_{\infty}$-structure on the graded vector space: 
\begin{equation}
\label{A-ve-graded}
A [[\ve,  \ve_1, \dots, \ve_g]]
\end{equation}
with the multiplications: 
$$
\sm^{\mu}_n (a_1, \dots, a_n) : =
\begin{cases}
\displaystyle  a_1 a_2   + \sum_{k_0 d_0 + \dots + k_g d_g = 0} 
\ve^{k_0} \ve_1^{k_1} \dots \ve_g^{k_g}
\mu_{k_0, k_1, \dots, k_g}  (a_1, a_2) \qquad {\rm if} ~~ n = 2, \\[0.8cm]
\displaystyle  \sum_{k_0 d_0 + \dots + k_g d_g = 2-n} 
\ve^{k_0} \ve_1^{k_1} \dots \ve_g^{k_g}
\mu_{k_0, k_1, \dots, k_g}  (a_1, \dots, a_n)  \qquad {\rm if} ~~ n > 2\,.
\end{cases} 
$$
Since the degrees of all formal parameters are non-positive, all non-zero 
$A_{\infty}$-multiplications have $\ge 2$ inputs. In other words, $\mu$ gives us 
a usual (i.e. non-curved) $A_{\infty}$-structure on \eqref{A-ve-graded} with the zero differential.

Just as in the $1$-parameter case, the MC equation for $\mu$ implies that the element
\begin{equation}
\label{ka-repres}
\ve \mu_{1, 0, \dots, 0} +  \ve_1 \mu_{0,1, 0, \dots, 0} + \dots + \ve_g  \mu_{0, \dots, 0, 1}
\end{equation}
is a degree $1$-cocycle in 
\begin{equation}
\label{repres-in}
\ve \Cbu(A) \oplus \ve_1 \Cbu(A) \oplus \dots \oplus \ve_g \Cbu(A) 
~\cong ~ \mm \Cbu(A) [[\ve,  \ve_1, \dots, \ve_g]] ~\big/~  \mm^2 \Cbu(A) [[\ve,  \ve_1, \dots, \ve_g]].
\end{equation}
Furthermore, isomorphic MC elements have cohomologous cocycles in \eqref{repres-in}.
As in the $1$-parameter case, we call the cohomology class $\ka$ of \eqref{ka-repres} the Kodaira-Spencer 
class of $\mu$.

%
%
\subsubsection*{The Penkava-Schwarz example} 
Let $A$ be the polynomial algebra $\bbC[x^1, \dots, x^{2n-1}]$ is $2n-1$ variables (of degree zero)
and $\ve_1$ be formal parameter of degree $3-2n$. 
Then the element
\begin{equation}
\label{example-mu}
\mu : = \ve_1 \del_{x^1} \cup  \del_{x^2} \cup \dots \cup \del_{x^{2n-1}}  ~\in ~ \mm \Cbu(A) [[\ve,  \ve_1, \dots, \ve_g]]
\end{equation}
is $\del^{\Hoch}$-closed. Furthermore,  
$$
[\mu, \mu]_{G} = 0
$$ 
since $\ve_1$ is an odd variable and hence $\ve_1^2 = 0$.    

Thus $\mu$ is a MC element of \eqref{Cbu-A-ve-graded} which gives an example 
of a deformation of $A$ in ``the $A_{\infty}$-direction''.  It is easy to see that the Kodaira-Spencer 
class of $\mu$ is non-zero. So this is an example of non-trivial deformation\footnote{A very similar
example is described in \cite[Section 3.2]{Keller} and \cite{P-Schwarz}.}.

\subsection{The case when $A$ is the algebra of functions $\cO(M)$ on a smooth manifold $M$}
\label{sec:case-cO-M} 

The general story presented above applies to the case $A = \cO(M)$ with the minor amendment: instead of the 
full Hochschild cochain complex, we use the sub- dg Lie algebra
\begin{equation}
\label{sPD-M}
\sPD^{\bul}(M)  \subset \Cbu(\cO(M)), 
\end{equation}
where $\sPD^{\bul}$ is the sheaf of polydifferential operators on $M$. 

By analogy with the $1$-parameter ``ungraded'' case (when $g = 0$) formal deformations of $\cO(M)$ 
with the base ring  $\bbC[[\ve, \ve_1, \dots, \ve_g]]$ are defined as MC elements $\mu$ of the dg Lie algebra 
\begin{equation}
\label{sPD-veveve}
\cL_{\sPD} : = \mm \sPD^{\bul}(M)[[\ve, \ve_1, \dots, \ve_g]].
\end{equation}
Furthermore, the equivalence classes of such deformations are elements of 
$$
\pi_0 \big( \cG(\cL_{\sPD})  \big),
$$
where the dg Lie algebra $\cL_{\sPD}$ is considered with the filtration 
\begin{equation}
\label{filtr-sPD-veveve}
\cF_k \cL_{\sPD} : = \mm^k \sPD^{\bul}(M)[[\ve, \ve_1, \dots, \ve_g]]. 
\end{equation}

The MC equation for $\mu$
\begin{equation}
\label{MC-mu}
\del^{\Hoch} \mu + \frac{1}{2} [\mu, \mu]_G  = 0
\end{equation}
implies that the coset of $\mu$ in 
$$
\cF_1 \cL_{\sPD} ~\big/~ \cF_2 \cL_{\sPD}
 ~ \cong ~
\ve \sPD^{\bul}(M) \oplus \ve_1 \sPD^{\bul}(M) \oplus \ve_2 \sPD^{\bul}(M) 
\oplus \dots \oplus \ve_g \sPD^{\bul}(M)
$$
is $\del^{\Hoch}$-closed and the corresponding vector in
\begin{equation}
\label{HH-veveve}
\ve \sPV^{1}(M) \oplus \ve_1  \sPV^{1-d_1}(M) \oplus \dots \oplus \ve_g  \sPV^{1-d_g}(M)
\end{equation} 
does not depend on the choice of a representative in the equivalence class of the deformation. 
We call the corresponding vector in \eqref{HH-veveve} the Kodaira-Spencer class of $\mu$. 

Since the degrees of all formal parameters are non-positive, the $A_{\infty}$-structure on  
\begin{equation}
\label{cO-ve-graded}
\cO(M)[[\ve,  \ve_1, \dots, \ve_g]]
\end{equation}
corresponding to $\mu$ has the following $\bbC[[\ve,  \ve_1, \dots, \ve_g]]$-multilinear multiplications: 
$$
\sm^{\mu}_n (a_1, \dots, a_n) : =
\begin{cases}
\displaystyle  a_1 a_2   + \sum_{k_0 d_0 + \dots + k_g d_g = 0} 
\ve^{k_0} \ve_1^{k_1} \dots \ve_g^{k_g}
\mu_{k_0, k_1, \dots, k_g}  (a_1, a_2) \qquad {\rm if} ~~ n = 2, \\[0.8cm]
\displaystyle  \sum_{k_0 d_0 + \dots + k_g d_g = 2-n} 
\ve^{k_0} \ve_1^{k_1} \dots \ve_g^{k_g}
\mu_{k_0, k_1, \dots, k_g}  (a_1, \dots, a_n)  \qquad {\rm if} ~~ n > 2\,,
\end{cases} 
$$
where $a_i \in \cO(M)$ and $\mu_{k_0, k_1, \dots, k_g}$ are the coefficients in the expansion of $\mu$
$$ 
\mu ~~ = ~~ \sum_{k_0 + k_1 + \dots + k_g \ge 1} 
\ve^{k_0} \ve_1^{k_1} \dots \ve_g^{k_g}
\mu_{k_0, k_1, \dots, k_g}\,.
$$

\subsection{The main result}
\label{sec:main}
Let $M$ be a smooth real manifold equipped with a symplectic structure $\om$ and 
$\al \in \sPV^1(M)$ be the corresponding (non-degenerate) Poisson structure. 
In other words, for every open coordinate subset $U \subset M$, we have
$$
\al^{ij}(x) \om_{jk}(x) = \de^i_k\,,
$$ 
where $\al^{ij}(x)$ (resp. $\om_{ij}(x)$) are components of $\al \big|_{U}$ (resp. $\om\big|_{U}$). 

Let us consider MC elements $\mu$ in \eqref{sPD-veveve}
satisfying these two conditions: 
\begin{cond}
\label{cond:KS}
The Kodaira-Spencer class of $\mu$ equals $\ve \al$.
\end{cond}
\begin{cond}
\label{cond:mu-ve}
The MC element $\mu$ satisfies the equation 
\begin{equation}
\label{cond-mu-ve}
\mu \big |_{\ve = 0} ~ = ~0.
\end{equation}
\end{cond}

We denote by $\ti{\cG}(\cL_{\sPD})$ the full subgroupoid of $\cG(\cL_{\sPD})$ 
whose objects are MC elements $\mu$ satisfying Conditions \ref{cond:KS} and 
\ref{cond:mu-ve}. Furthermore, we denote by $\TL$ the set of isomorphism 
classes of objects of $\ti{\cG}(\cL_{\sPD})$, i.e.
\begin{equation}
\label{TL}
\TL : = \pi_0 \big( \ti{\cG}(\cL_{\sPD})  \big).
\end{equation}
We call $\TL$ the {\it topological locus} of $\pi_0 \big( \cG(\cL_{\sPD})  \big)$. 

\begin{remark}
\label{rem:cond}
Note that every MC element $\mu$ satisfying Condition \ref{cond:mu-ve} is isomorphic 
to infinitely many MC elements of $\cL_{\sPD}$ which do not satisfy this condition. 
Indeed, consider a MC element $\mu$ which satisfies \eqref{cond-mu-ve} and 
a polydifferential operator $P \in \sPD^{- d_1}(M)$ 
for which $\del^{\Hoch} P \neq 0$. Then the MC element 
$$
e^{[\ve_1 P, ~]_G} \, \mu - \frac{e^{[\ve_1 P, ~]_G} - 1}{[\ve_1 P,~]_G}\, (\del^{\Hoch}\, \ve_1 P)
$$
does not satisfy equation \eqref{cond-mu-ve}. 
\end{remark}

Equation \eqref{cond-mu-ve} guarantees that the $A_{\infty}$-multiplications $\{ \sm_n \}_{n \ge 2}$
corresponding to the MC element $\mu$ satisfy the property
\begin{equation}
\label{mult-ve-0}
\sm_n(a_1, \dots, a_n) \big|_{\ve = 0} = 
\begin{cases}
a_1 a_2 \qquad {\rm if} ~~  n = 2 \\
0 \qquad ~~~~ {\rm otherwise}.
\end{cases}  
\end{equation}
In other words, the $A_{\infty}$-algebra 
$$
\big( \cO(M)[[\ve, \ve_1, \dots, \ve_g]], \{ \sm_n \}_{n \ge 2} \big)
$$
can be viewed as a $1$-parameter formal deformation of the graded 
commutative algebra $\cO(M)[[\ve_1, \dots, \ve_g]]$.

The goal of this note is to describe the topological locus $\TL$ of equivalence 
classes of formal deformations of $\cO(M)$ with the base ring $\bbC[[\ve, \ve_1, \dots, \ve_g]]$:
\begin{thm} 
\label{thm:main}
For every symplectic manifold $M$, the isomorphism classes of formal 
deformations of $\cO(M)$ with the base ring $\bbC[[\ve, \ve_1, \dots, \ve_g]]$ satisfying 
Conditions \ref{cond:KS} and \ref{cond:mu-ve} are in bijection with elements of the vector space 
\begin{equation}
\label{veveve-H}
\bigoplus_{q \geq 0} \frac{1}{\ve^{q-1}} \, \big( \mm \, \bs^q\, H^q(M, \bbC)[[\ve, \ve_1, \dots, \ve_g]] \big)^2\,.
\end{equation}
Here $\mm$ is the maximal ideal $(\ve, \ve_1, \dots, \ve_g) \subset \bbC[\ve, \ve_1, \dots, \ve_g]$,
$H^{\bul}(M, \bbC)$ is the singular cohomology of $M$ with coefficients in $\bbC$, and 
$\big( \mm \, \bs^q\, H^q(M, \bbC)[[\ve, \ve_1, \dots, \ve_g]] \big)^2$ is the subspace of degree $2$ elements in 
the graded vector space 
$$
\mm \, \bs^q\, H^q(M, \bbC)[[\ve, \ve_1, \dots, \ve_g]].
$$
\end{thm}

The proof of Theorem \ref{thm:main} is given in Section \ref{sec:proof} and it 
is based on two auxiliary constructions. The first construction is presented in 
Section \ref{sec:PV-PD} and its main ingredient is Kontsevich's formality 
quasi-isomorphism \cite{K} for polydifferential operators. The second construction
is presented in Section \ref{sec:PV} and it is inspired by a result \cite{ShT} due to 
G. Sharygin and D. Talalaev.

\section{Applying Kontsevich's formality theorem}
\label{sec:PV-PD}

Let us fix an $L_{\infty}$-quasi-isomorphism 
\begin{equation}
\label{cU-fixed}
\cU :   \sPV^{\bul}(M) \leadsto  \sPD^{\bul}(M) 
\end{equation}
whose linear term coincides with the embedding \eqref{HKR-M}. 
Let us also extend it by  $\bbC[[\ve, \ve_1, \dots, \ve_g]]$-linearity to 
\begin{equation}
\label{cU-veveve}
\cU :  \mm\, \sPV^{\bul}(M)[[\ve, \ve_1, \dots, \ve_g]] \leadsto \mm  \sPD^{\bul}(M)[[\ve, \ve_1, \dots, \ve_g]] 
\end{equation}
and denote by $\mu_{\al}$ the following MC element of \eqref{sPD-veveve} 
\begin{equation}
\label{mu-al}
\mu_{\al} ~: =~ \cU_*(\ve \al).
\end{equation}

Twisting \eqref{cU-veveve} by $\al$, we get an $L_{\infty}$-morphism 
\begin{equation}
\label{cU-al}
\cU^{\al} :   \big( \mm \sPV^{\bul}(M)[[\ve, \ve_1, \dots, \ve_g]], [\ve\al, ~]_S  \big)
\leadsto  \cL^{\mu_{\al}}_{\sPD}\,,
\end{equation}
where the dg Lie algebra $\cL^{\mu_{\al}}_{\sPD}$ is obtained from $\cL_{\sPD}$ via replacing 
the differential $\del^{\Hoch}$ by
$$
\del^{\Hoch} + [\mu_{\al},~]_G\,. 
$$ 

Notice that 
$$
\del^{\Hoch} + [\mu_{\al},~]_G = \del^{\Hoch}_*\,, 
$$
where $\del^{\Hoch}_*$ is the Hochschild differential corresponding to the star product
\begin{equation}
\label{star-alpha}
a * b : = a b + \mu_{\al}(a,b).
\end{equation}
Furthermore, the dg Lie algebra $\cL^{\mu_{\al}}_{\sPD}$ carries the 
same descending filtration as $\cL_{\sPD}$
\begin{equation}
\label{filtr-sPD-twisted}
\cF_k \cL^{\mu_{\al}}_{\sPD} : =  \mm^k \sPD^{\bul}(M)[[\ve, \ve_1, \dots, \ve_g]].
\end{equation}

It is easy to see that the formula 
\begin{equation}
\label{Shift-dfn}
\Shift_{\mu_{\al}} (\ti{\mu}) : = \mu_{\al} +  \ti{\mu} 
\end{equation}
defines a bijection from the set of MC elements of $\cL^{\mu_{\al}}_{\sPD}$ 
to the set of MC elements of $\cL_{\sPD}$ \eqref{sPD-veveve}. 
A simple direct computation\footnote{For the version of this statement in the setting 
of $L_{\infty}$-algebras and the corresponding DGH $\infty$-groupoids, we refer the reader to 
\cite[Lemma 4.3]{enhanced}.} shows that for every degree zero element $\xi \in \cL^{\mu_{\al}}_{\sPD}$
and for every  $\ti{\mu} \in \MC(\cL^{\mu_{\al}}_{\sPD})$
$$
\Shift_{\mu_{\al}} \Big( \,
e^{[\xi, ~]_G} \ti{\mu} - \frac{e^{[\xi, ~]_G} - 1}{[\xi,~]_G} (\del^{\Hoch} \xi + [\mu_{\al},\xi]_G) \, \Big) ~=~
e^{[\xi, ~]_G} (\mu_{\al} +\ti{\mu}) ~ - ~\frac{e^{[\xi, ~]_G} - 1}{[\xi,~]_G} (\del^{\Hoch} \xi).
$$
In other words, $\Shift_{\mu_{\al}}$ upgrades to a functor 
\begin{equation}
\label{Shift-functor}
\Shift_{\mu_{\al}} : \cG( \cL^{\mu_{\al}}_{\sPD} ) \to \cG( \cL_{\sPD})
\end{equation}
which acts ``as identity'' on the set of morphisms. It is not hard to see that 
\eqref{Shift-functor} is actually a \und{strict} isomorphism of groupoids and the inverse 
functor operates on objects as
\begin{equation}
\label{Shift-inv}
\mu \mapsto \mu - \mu_{\al} ~: ~\MC(\cL_{\sPD}) \to \MC(\cL^{\mu_{\al}}_{\sPD}).
\end{equation}

Let us denote by $\ti{\cL}_{\sPD}$ the following sub- dg Lie algebra of $\cL^{\mu_{\al}}_{\sPD}$
\begin{equation}
\label{ve-mm-sPD}
\ti{\cL}_{\sPD} : = \big( \ve\, \mm \, \sPD^{\bul}(M)[[\ve, \ve_1, \dots, \ve_g]], \, \del^{\Hoch}_*,\,  [~,~]_G \big) 
\end{equation}
and observe that the filtration from $\cL^{\mu_{\al}}_{\sPD}$ induces the descending filtration on $\ti{\cL}_{\sPD}$: 
\begin{equation}
\label{tilde-sPD-filtr}
\cF_k \ti{\cL}_{\sPD} : = \cF_k \cL^{\mu_{\al}}_{\sPD} \cap \ti{\cL}_{\sPD} = \ve \, \mm^{k-1} \, \sPD^{\bul}(M)[[\ve, \ve_1, \dots, \ve_g]].
\end{equation}

Next, we denote by $\ti{\cG}(\cL^{\mu_{\al}}_{\sPD})$ the full subgroupoid of $\cG(\cL^{\mu_{\al}}_{\sPD})$
whose set of objects is $\MC(\ti{\cL}_{\sPD})$. Moreover, we set 
\begin{equation}
\label{TL-tw}
\TL^{\tw} : = \pi_0 \big( \ti{\cG}(\cL^{\mu_{\al}}_{\sPD}) \big).
\end{equation}

Let us prove that 
\begin{prop}
\label{prop:MC-in-quest}
The restriction of the functor \eqref{Shift-functor}
to the subgroupoid $\ti{\cG}(\cL^{\mu_{\al}}_{\sPD})$ 
induces a bijection 
$$
\TL^{\tw} \cong \TL\,.
$$
\end{prop}
\begin{proof} Let $\ti{\mu}$ be a MC element of $\ti{\cL}_{\sPD}$. 
It is clear that the Kodaira-Spencer class of $\mu_{\al} + \ti{\mu}$ coincides 
with the Kodaira-Spencer class of $\mu_{\al}$. Hence  $\mu_{\al} + \ti{\mu}$ satisfies 
Condition \ref{cond:KS}. 

Since 
$$
\mu_{\al} \big|_{\ve =0}  ~ = ~ \ti{\mu} \big|_{\ve =0} ~=~ 0, 
$$
the MC element $\mu_{\al} + \ti{\mu}$ also satisfies Condition \ref{cond:mu-ve}. 

Thus restricting $\Shift_{\mu_{\al}}$ to the full sub-groupoid $\ti{\cG}(\cL^{\mu_{\al}}_{\sPD})$,  
we get a functor 
\begin{equation}
\label{Shift-restr}
\ti{\cG}(\cL^{\mu_{\al}}_{\sPD}) \to \ti{\cG}( \cL_{\sPD}).
\end{equation}
Hence we get a map 
\begin{equation}
\label{TL-tw-TL}
\TL^{\tw} \to \TL.
\end{equation}

Using the functor \eqref{Shift-inv} from $\cG(\cL_{\sPD})$ to $\cG(\cL^{\mu_{\al}}_{\sPD})$, it is easy 
show that the map \eqref{TL-tw-TL} is one-to-one. So it remains to show that, for every 
$\mu \in \MC(\cL_{\sPD})$ satisfying Conditions \ref{cond:KS} and \ref{cond:mu-ve}, there exists 
$\ti{\mu} \in \MC(\ti{\cL}_{\sPD})$ such that $\mu$ is isomorphic to  $\mu_{\al} + \ti{\mu}$. 

Since the Kodaira-Spencer classes of $\mu$ and $\mu_{\al}$ coincide, the coset of the difference $\mu - \mu^{\al}$ in 
$$
\mm \sPD^{\bul}(M)[[\ve, \ve_1, \dots, \ve_g]] ~\big/~ \mm^2 \sPD^{\bul}(M)[[\ve, \ve_1, \dots, \ve_g]]
$$
is $\del^{\Hoch}$-exact. 

Hence there exists a degree zero vector $\xi \in \mm \sPD^{\bul}(M)[[\ve, \ve_1, \dots, \ve_g]]$ such that 
\begin{equation}
\label{exp-xi-mu-mu-al}
\Big( e^{[\xi, ~]_G} \mu - \frac{e^{[\xi, ~]_G} - 1}{[\xi,~]_G} (\del^{\Hoch}\, \xi) \Big) ~ - ~ \mu_{\al} \in  \mm^2 \sPD^{\bul}(M)[[\ve, \ve_1, \dots, \ve_g]].
\end{equation}
Moreover, since $\mu \big|_{\ve = 0} = 0$, the vector $\xi$, for which \eqref{exp-xi-mu-mu-al} holds, can be 
found in 
$$
(\ve) \sPD^{\bul}(M)[[\ve, \ve_1, \dots, \ve_g]]. 
$$

Therefore, 
$$
\Big( e^{[\xi, ~]_G} \mu - \frac{e^{[\xi, ~]_G} - 1}{[\xi,~]_G} (\del^{\Hoch}\, \xi) \Big) ~ - ~ \mu_{\al} \in  \ve \, \mm \sPD^{\bul}(M)[[\ve, \ve_1, \dots, \ve_g]].
$$
In other words, the MC element 
$$
\ti{\mu} : = \Big( e^{[\xi, ~]_G} \mu - \frac{e^{[\xi, ~]_G} - 1}{[\xi,~]_G} (\del^{\Hoch}\, \xi) \Big) ~ - ~ \mu_{\al}
$$
of $\cL^{\mu_{\al}}_{\sPD}$ belongs to the sub- dg Lie algebra $\ti{\cL}_{\sPD}$. 

The MC element $\mu_{\al} + \ti{\mu}$ of $\cL_{\sPD}$ is isomorphic to $\mu$ by construction. 

Thus the proposition is proved. 
\end{proof}

\bigskip \bigskip

To describe the set $\TL^{\tw}$, we denote by $\cL_{\sPV}$ and $\ti{\cL}_{\sPV}$ the dg Lie algebra 
$$
\cL_{\sPV} : = \big(\mm \, \sPV^{\bul}(M)[[\ve, \ve_1, \dots, \ve_g]], \, [\ve \al, ~]_S ,\,  [~,~]_S \big)
$$
and its sub- dg Lie algebra 
\begin{equation}
\label{ve-mm-sPV}
\ti{\cL}_{\sPV} : =  \ve\, \mm \, \sPV^{\bul}(M)[[\ve, \ve_1, \dots, \ve_g]] ~ \subset ~ \cL_{\sPV}\,,
\end{equation}
respectively. 

We consider $\cL_{\sPV}$ and $\ti{\cL}_{\sPV}$ with the following descending filtrations: 
\begin{equation}
\label{filtr-mm-sPV}
\cF_k \cL_{\sPV} : = \mm^k \, \sPV^{\bul}(M)[[\ve, \ve_1, \dots, \ve_g]],
\end{equation}
\begin{equation}
\label{filtr-ve-mm-sPV}
\cF_k \ti{\cL}_{\sPV} ~: = ~ \ti{\cL}_{\sPV} \cap \cF_k \cL_{\sPV} ~ = ~
\ve\, \mm^{k-1} \, \sPV^{\bul}(M)[[\ve, \ve_1, \dots, \ve_g]].
\end{equation}

Furthermore, we denote by $\ti{\cG}(\cL_{\sPV})$ the full subgroupoid of $\cG(\cL_{\sPV})$ 
whose set of objects is $\MC(\ti{\cL}_{\sPV})$ and set 
\begin{equation}
\label{TL-sPV}
\TL_{\sPV} : = \pi_0\big( \ti{\cG}(\cL_{\sPV})\big).
\end{equation}

Restricting $\cU^{\al}$ \eqref{cU-al} to the sub- dg Lie algebra $\ti{\cL}_{\sPV}$, we get an $L_{\infty}$-morphism 
\begin{equation}
\label{cU-al-needed}
\cU^{\al} : \ti{\cL}_{\sPV} \leadsto \ti{\cL}_{\sPD}\,.
\end{equation}

Let us prove that 
\begin{claim}
\label{cl:GM-OK}
The $L_{\infty}$-morphism \eqref{cU-al} (resp. \eqref{cU-al-needed}) is  
compatible with the filtrations \eqref{filtr-sPD-twisted}, \eqref{filtr-mm-sPV}
(resp. \eqref{tilde-sPD-filtr}, \eqref{filtr-ve-mm-sPV}) in the sense of \eqref{compat-w-filtr}.
Moreover the linear terms of the $L_{\infty}$-morphisms \eqref{cU-al}
and \eqref{cU-al-needed} give us quasi-isomorphisms of cochain complexes
$$
\cF_k \cL_{\sPV} \stackrel{\sim}{~\longrightarrow~} \cF_k \cL^{ \mu_{\al} }_{\sPD}\,,
$$
$$
\cF_k \ti{\cL}_{\sPV} \stackrel{\sim}{~\longrightarrow~} \cF_k \ti{\cL}_{\sPD}
$$
for every $k \ge 1$, respectively. 
\end{claim}
\begin{proof}
The $L_{\infty}$-morphisms \eqref{cU-al} and \eqref{cU-al-needed} are compatible with 
the filtrations by construction. So we proceed to the second statement.

For every fixed $k \ge 1$ the dg Lie algebras $\cF_k \cL_{\sPV}$, $\cF_k \ti{\cL}_{\sPV}$, 
$\cF_k \cL^{ \mu_{\al} }_{\sPD}$, and $\cF_k \ti{\cL}_{\sPD}$ are equipped with the complete 
descending filtrations. For example, 
$$
\cF_k \cL_{\sPV} \supset \cF_{k+1} \cL_{\sPV} \supset \cF_{k+2} \cL_{\sPV} \supset \dots
$$

The linear term $\cU^{\al}_1$ gives us chain maps
\begin{equation}
\label{no-tilde}
(\cF_k \cL_{\sPV}, [\ve \al,~ ]_S ) ~\to~ (\cF_k \cL^{ \mu_{\al} }_{\sPD}, \del^{\Hoch}_*)
\end{equation}
and
\begin{equation}
\label{yes-tilde}
(\cF_k \ti{\cL}_{\sPV}, [\ve \al,~ ]_S ) ~\to~ (\cF_k \ti{\cL}_{\sPD}, \del^{\Hoch}_*)
\end{equation}
compatible with these filtrations. As above, $\del^{\Hoch}_*$ is the Hochschild differential 
corresponding to the star product \eqref{star-alpha}. 

At the level of associated graded complexes, we get the chain maps 
\begin{equation}
\label{no-tilde-Gr}
J_{HKR} : \big( \mm^k \,\sPV^{\bul}(M)[\ve, \ve_1, \dots, \ve_g] , 0 \big) ~\to~ 
\big( \mm^k\, \sPD^{\bul}(M)[\ve, \ve_1, \dots, \ve_g], \del^{\Hoch} \big)
\end{equation} 
and
\begin{equation}
\label{yes-tilde-Gr}
J_{HKR} : \big( \ve \mm^{k-1} \,\sPV^{\bul}(M)[\ve, \ve_1, \dots, \ve_g] , 0 \big) ~\to~ 
\big( \ve \mm^{k-1}\, \sPD^{\bul}(M)[\ve, \ve_1, \dots, \ve_g], \del^{\Hoch} \big)\,,
\end{equation}
where $\del^{\Hoch}$ is the Hochschild differential corresponding to the usual (commutative)
multiplication on $\cO_M$ and $J_{HKR}$ is the Hochschild-Kostant-Rosenberg embedding 
\eqref{HKR} extended by linearity with respect to $\bbC[\ve, \ve_1, \dots, \ve_g]$.

Since the functors $\otimes_{\bbC}\, \mm^k$ and  $\otimes_{\bbC}\, \ve \mm^{k-1}$
preserve quasi-isomorphisms of cochain complexes, Proposition \ref{prop:HKR} implies 
that \eqref{no-tilde-Gr} and \eqref{yes-tilde-Gr} are quasi-isomorphisms. 

Hence, by the Lemma on filtered complexes (see, for example, \cite[Lemma D.1, App. D]{Deligne-TW}), 
the chain maps \eqref{no-tilde} and \eqref{yes-tilde} are quasi-isomorphisms of cochain complexes. 

Claim \ref{cl:GM-OK} is proved. 
\end{proof}

Let us prove that 
\begin{prop}
\label{prop:to-PV}
The $L_{\infty}$-morphism \eqref{cU-al-needed} induces a map 
\begin{equation}
\label{TL-PV-wt-TL}
\TL_{\sPV}  \to  \TL^{\tw}\,.
\end{equation}
Furthermore, this map is a bijection. 
\end{prop}
\begin{proof}
To prove the first statement, we observe that the $L_{\infty}$-morphism \eqref{cU-al-needed}
gives a map 
\begin{equation}
\label{cU-MC}
\cU^{\al}_* : \MC( \ti{\cL}_{\sPV}) \to \MC(\ti{\cL}_{\sPD}). 
\end{equation}
Let  $\mu_1$ and $\mu_2$ be MC elements of $\ti{\cL}_{\sPV}$ 
connected by a $1$-cell 
$$
\eta \in  \mMC_{1} (\cL_{\sPV}) = \MC \big( \cL_{\sPV} \hotimes \bbC[t] \oplus  \cL_{\sPV} \hotimes \bbC[t]\, d t \big),
$$
i.e.
$$
\eta\big|_{t = dt = 0} = \mu_1, \qquad \eta\big|_{t = 1, dt = 0} = \mu_2.
$$

Clearly, the $1$-cell
$$
\cU^{\al}_* (\eta) ~ \in~   \mMC_{1} (\cL_{\sPD})
$$
connects the MC elements 
$$
\cU^{\al}_*(\mu_1),~ \cU^{\al}_*(\mu_2)~ \in ~\ti{\cL}_{\sPD}. 
$$

Thus \eqref{cU-MC} descends to the map of sets 
$$
\TL_{\sPV}  \to  \TL^{\tw}\,.
$$

Let us now show that this map is a bijection. 

According to Claim \ref{cl:GM-OK}, both $L_{\infty}$-morphisms  \eqref{cU-al}
and \eqref{cU-al-needed} satisfy conditions of \cite[Theorem 1.1]{GM}. 
So, applying this theorem to \eqref{cU-al-needed}, we conclude that 
the map \eqref{TL-PV-wt-TL} is surjective. 
Furthermore, applying  \cite[Theorem 1.1]{GM} to \eqref{cU-al}, we conclude that 
the map \eqref{TL-PV-wt-TL} is injective. 
\end{proof}

\section{Passing to exterior forms}
\label{sec:PV}

Let $f \in \cO(M)$ and $\te_i : = \del_{x^i}$ be the local vector field 
defined in a coordinate chart of $M$. Since the symplectic structure $\om$
is non-degenerate,  the formulas
$$
\sJ_{\om} (f) : = f, \qquad 
\sJ_{\om} (\te_i) : = \frac{1}{\ve} \om_{ij}(x) d x^j 
$$
define an isomorphism of (shifted) graded commutative algebras
\begin{equation}
\label{J-om}
\sJ_{\om} :  \mm \, \sPV(M)  [[\ve, \ve_1, \dots, \ve_g]] ~\to~ 
\bigoplus_{q \ge 0} \frac{1}{\ve^{q}}  \bsi \, \mm \, \Om^q(M) [[\ve, \ve_1, \dots, \ve_g]]. 
\end{equation}
The inverse of $\sJ_{\om}$ is given by
\begin{equation}
\label{J-om-inv}
\sJ^{-1}_{\om} (f) = f, \qquad 
\sJ^{-1}_{\om} (dx^i)  = \ve \al^{ij}(x) \te_j\,,
\end{equation}
where $\al^{ij}(x)$ are components of the corresponding Poisson structure (in local coordinates). 

Here, we tacitly assume that every vector $\eta \in \Om^q(M)$ carries degree $q$.
The latter means that the vector $\eta\, \ve^{k_1}_1 \dots  \ve^{k_g}_g$ carries degree 
$q + k_1 d_1 + \dots + k_g d_g$ and the vector $\bsi \eta \, \ve^{k_1}_1 \dots  \ve^{k_g}_g$ carries degree 
$q -1 + k_1 d_1 + \dots + k_g d_g$.

Using this isomorphism and the dg Lie algebra structure $([\ve \al, ~]_S, [~,~]_S)$ on $\cL_{\sPV}$, 
we equip the graded vector space 
\begin{equation}
\label{cL-Om}
\cL_{\Om} : = \bigoplus_{q \ge 0} \frac{1}{\ve^{q}}  \bsi \,\mm \, \Om^q(M) [[\ve, \ve_1, \dots, \ve_g]] 
\end{equation}
with the structure of a dg Lie algebra. 

A direct computation shows that the differential on $\cL_{\Om}$ corresponding to $[\ve \al, ~]_S$ is 
$$
-d
$$
(where $d$ is the de Rham differential) and 
the Lie bracket $[~,~]_{\om}$ corresponding to $[~,~]_S$ is given 
by the formulas (in local coordinates)
$$
[\bsi \eta_1,  \bsi \eta_2]_{\om} : =
$$
\begin{equation}
\label{brack-on-forms}
 \bsi \ve dx^k \del_{x^k} \al^{ij}(x) 
\frac{\del \, \eta_1}{\del d x^i} \frac{\del \, \eta_2}{\del d x^j} 
-(-1)^{|\eta_1|} \bsi \ve \al^{ij}(x) \frac{\del \, \eta_1}{\del d x^i}\, \del_{x^j} \eta_2 
+ \bsi \ve \al^{ij}(x) (\del_{x^i} \eta_1) \frac{\del \, \eta_2}{ \del d x^j}\,.
\end{equation}
For example, 
\begin{equation}
\label{brack-on-forms-easy}
[\bsi f, \bsi g]_{\om} : = 0, \qquad 
[\bsi d x^i, \bsi f]_{\om} : = \bsi \ve\, \al^{ij}(x) (\del_{x^j} f), 
\end{equation}
$$
[\bsi d x^i,  \bsi d x^j]_{\om} : =   \bsi \ve \, d x^k \del_{x^k} \al^{ij}(x),
$$
where $f,g \in \cO(M)$.

It is clear that, restricting $\sJ_{\om}$ to the sub- dg Lie algebra 
\eqref{ve-mm-sPV}, we get an isomorphism of dg Lie algebras 
\begin{equation}
\label{ticL-PV-ticL-Om}
\sJ_{\om} :  \ti{\cL}_{\sPV} \stackrel{\cong}{\longrightarrow}  \ti{\cL}_{\Om}\,,
\end{equation}
where $\ti{\cL}_{\Om}$ is the sub- dg Lie algebra of $\cL_{\Om}$: 
\begin{equation}
\label{ti-cL-Om}
\ti{\cL}_{\Om} : = \bigoplus_{q \ge 0} \frac{1}{\ve^{q-1}}  \bsi \,\mm \, \Om^q(M) [[\ve, \ve_1, \dots, \ve_g]]. 
\end{equation}

Both dg Lie algebras \eqref{cL-Om} and \eqref{ti-cL-Om} are equipped with the obvious 
complete descending filtrations
\begin{equation}
\label{filtr-cL-Om}
\cF_k \cL_{\Om} : = \bigoplus_{q \ge 0} \frac{1}{\ve^{q}}  \bsi \,\mm^k \, \Om^q(M) [[\ve, \ve_1, \dots, \ve_g]], 
\end{equation}
\begin{equation}
\label{filtr-ti-cL-Om}
\cF_k \ti{\cL}_{\Om} : = \bigoplus_{q \ge 0} \frac{1}{\ve^{q-1}}   \bsi \,\mm^{k-1} \, \Om^q(M) [[\ve, \ve_1, \dots, \ve_g]]. 
\end{equation}
Moreover, the isomorphisms \eqref{J-om} and \eqref{ticL-PV-ticL-Om} are compatible with these filtrations. 

We denote by $\ti{\cG}(\cL_{\Om})$ the full subgroupoid of $\cG(\cL_{\Om})$ whose 
set of objects is $\MC(\ti{\cL}_{\Om})$ and set 
\begin{equation}
\label{TL-Om}
\TL_{\Om} : = \pi_0 \big( \ti{\cG}(\cL_{\Om}) \big).
\end{equation}

Using the above properties of the isomorphism $J_{\om}$ we easily deduce that 
\begin{claim}
\label{cl:TL-Om}
The map $J_{\om}$ gives us an isomorphism of groupoids
$$
(J_{\om})_* :  \ti{\cG}(\cL_{\sPV})  \stackrel{\cong}{\, \longrightarrow \,}  \ti{\cG}(\cL_{\Om})
$$
and hence a bijection 
$$
\TL_{\sPV} \stackrel{\cong}{\, \longrightarrow \,} \TL_{\Om}\,.
$$
\qed
\end{claim} 
\begin{remark}
\label{rem:forms-shifted}
Note that the grading on $\cL_{\Om}$ and $\ti{\cL}_{\Om}$ comes with 
an additional shift. For example, every degree zero exterior form 
$\ve \, \eta \in \ve \Om^0(M) = \ve \cO(M)$ gives us the degree $-1$ 
vector $\bsi \ve\,\eta \in \cL_{\Om}$. In particular, monomials 
in $S^n(\bsi \ti{\cL}_{\Om})$ will be written as 
$$
\bs^{-2} \eta_1 \, \bs^{-2} \eta_2 \, \dots \, \bs^{-2} \eta_n,
$$
where $\eta_1, \eta_2, \dots, \eta_n$ belong to the space
\begin{equation}
\label{before-bsi}
\bigoplus_{q \ge 0} \frac{1}{\ve^{q-1}} \, \mm \, \Om^q(M) [[\ve, \ve_1, \dots, \ve_g]]. 
\end{equation}
\end{remark}

Let us consider the cocommutative coalgebra 
\begin{equation}
\label{S-ti-cL}
\und{S}(\bsi \ti{\cL}_{\Om})
\end{equation} 
with the degree $1$ coderivations $Q_{-d}$ and $Q_{[~,~]_{\om}}$ where 
$Q_{-d}$ (resp. $Q_{[~,~]_{\om}}$) comes from the dg Lie algebra structure 
$(-d, [~,~] = 0)$ (resp. $(\del =0, [~,~]_{\om})$) on $\ti{\cL}_{\Om}$  in the sense 
of \eqref{CE-Q}.  For example, 
\begin{equation}
\label{Q-d-exam}
Q_{-d} (\bs^{-2} \eta_1 \, \bs^{-2} \eta_2) = - \bsis d \eta_1 \,  \bsis \eta_2 - (-1)^{|\eta_1|} \bsis \eta_1  \,  \bsis d \eta_2\,, 
\end{equation}
and 
\begin{equation}
\label{Q-brack-exam}
Q_{[~,~]_{\om}} (\bs^{-2} \eta_1 \,  \bs^{-2} \eta_2) : = 
(-1)^{|\eta_1|} \bs^{-2}\, \ve \, dx^k \del_{x^k} \al^{ij}(x) \frac{\del \eta_1}{\del dx^i} \frac{\del \eta_2}{\del dx^j}
\end{equation}
$$
- \bs^{-2}\, \ve \, \al^{ij}(x) \frac{\del \eta_1}{\del d x^i} \del_{x^j} \eta_2
+ (-1)^{|\eta_1|}  \bs^{-2}\, \ve \, \al^{ij}(x) (\del_{x^i} \eta_1) \frac{\del \eta_2}{\del dx^j}\,,
$$
where $|\eta_1|$ is the degree of $\eta_1$ in \eqref{before-bsi}. 

Using the idea of \cite{ShT}, we consider the coderivation $\Pi$ of the coalgebra \eqref{S-ti-cL}
defined in local coordinates by the formulas
\begin{equation}
\label{Pi}
p \circ \Pi   (\bs^{-2} \eta_1 \, \bs^{-2} \eta_2 ) ~ : =~
(-1)^{|\eta_1|}  \bs^{-2}\, \ve  \al^{ij}(x) \Big( \frac{\del}{\del d x^i } \eta_1 \Big) \Big(\frac{\del}{\del d x^j } \eta_2 \Big)\,,    
\end{equation}
$$
p \circ \Pi   (\bs^{-2} \eta_1 \, \bs^{-2} \eta_2 \, \dots \,  \bs^{-2} \eta_n ) ~ : = 0, \qquad 
\textrm{if} ~~~ n \neq 2,
$$
where $p$ is the canonical projection $\und{S}( \bsi \ti{\cL}_{\Om} )  \to \bsi \ti{\cL}_{\Om}$. 

Properties of the coderivation $\Pi$ are listed in the following proposition:
\begin{prop}
\label{prop:Pi}
The coderivation $\Pi$ has degree $0$. Furthermore, 
\begin{equation}
\label{Pi-nice}
\Pi \circ Q_{-d}  - Q_{-d} \circ \Pi = Q_{[~,~]_{\om}}
\end{equation}
and 
\begin{equation}
\label{Pi-with-Q-bracket}
\Pi \circ  Q_{[~,~]_{\om}} =  Q_{[~,~]_{\om}} \circ \Pi.
\end{equation}
\end{prop}
\begin{remark}
\label{rem:Pi}
By keeping track of terms of negative powers of $\ve$, it is easy 
to see that, in general, 
$$
p \circ \Pi  \big(  \bsi \cL_{\Om} \otimes  \bsi \cL_{\Om} \big) ~\not\subset~  \bsi \cL_{\Om}\,.
$$
However 
\begin{equation}
\label{Pi-is-good}
p \circ \Pi  \big(  \bsi \ti{\cL}_{\Om} \otimes  \bsi \cL_{\Om} \big) ~\subset~  \bsi \cL_{\Om}
\end{equation}
and we will use the inclusion \eqref{Pi-is-good} to extend the coderivation $\Pi$ to 
the coalgebra 
\begin{equation}
\label{S-L}
\und{S}(\bsi L),
\end{equation}
where $L$ is the following graded vector space
\begin{equation}
\label{L-dfn}
L : =  \ti{\cL}_{\Om} \hotimes \bbC[t] ~\oplus~  \cL_{\Om} \hotimes \bbC[t] dt\,,
\end{equation}
$\hotimes$ is the completed tensor product and $\bbC[t] \oplus  \bbC[t] dt$ is the 
algebra of polynomial de Rham forms on the $1$-simplex. We will freely use the obvious 
generalizations of \eqref{Pi-nice} and \eqref{Pi-with-Q-bracket} to the corresponding 
coderivations of the coalgebra \eqref{S-L}. 
\end{remark}

The proof of Proposition \ref{prop:Pi} is given in Appendix \ref{app:Pi}. 
Here we use this proposition to deduce the following statement:
\begin{cor}
\label{cor:exp-Pi}
The formula 
\begin{equation}
\label{exp-Pi}
\exp(\Pi) : = 1 + \sum_{m \ge 1} \frac{1}{m!} \Pi^m
\end{equation}
defines a \und{strictly} invertible $L_{\infty}$ quasi-isomorphism 
$$
(\ti{\cL}_{\Om}, -d, 0)  \leadsto (\ti{\cL}_{\Om}, -d, [~,~]_{\om}). 
$$
This isomorphism extends, in the obvious way, to a strictly invertible 
$L_{\infty}$ quasi-isomorphism 
$$
(L, -d + d_t, 0) \leadsto (L, -d + d_t, [~,~]_{\om}),
$$
where $L$ is defined in \eqref{L-dfn} and $d_t$ is the de Rham differential $dt \del_t$ on
$\bbC[t] \,\oplus\, \bbC[t] \, dt$.
\end{cor}
\begin{proof} 
It is straightforward to show that the equation \eqref{exp-Pi} defines automorphisms
of the cocommutative coalgebras \eqref{S-ti-cL} and \eqref{S-L} (considered with the zero 
differentials). The inverse of  \eqref{exp-Pi} is given by the formula:
\begin{equation}
\label{exp-Pi-inv}
\exp(- \Pi) : = 1 + \sum_{m \ge 1} \frac{(-1)^m}{m!} \Pi^m\,.
\end{equation}

It remains to prove that 
$$
\exp(\Pi)\circ Q_{-d} = (Q_{-d} + Q_{[~,~]_{\om}}) \circ  \exp(\Pi)
$$
or equivalently
\begin{equation}
\label{exp-Pi-good}
\exp(\Pi)\circ Q_{-d} \circ  \exp(-\Pi) = Q_{-d} + Q_{[~,~]_{\om}}\,.
\end{equation}

For this purpose we introduce two elements 
$$
\Psi_L, \Psi_R \in \Hom\big(\und{S}(\bsi \ti{\cL}_{\Om})\,,\, \und{S}(\bsi \ti{\cL}_{\Om})[u] \big) 
$$
defined by 
\begin{equation}
\label{Psi-s}
\Psi_L : = \exp(u\, \Pi)\circ Q_{-d} \circ  \exp(-u\, \Pi),
\qquad 
\Psi_R : =  Q_{-d} +  u\, Q_{[~,~]_{\om}}\,,
\end{equation}
where $u$ is an auxiliary variable of degree $0$. 

Using \eqref{Pi-nice} and \eqref{Pi-with-Q-bracket} we get
$$
\frac{d}{du} \Psi_L = \exp(u \Pi) \circ  \big( \Pi \circ Q_{-d} - Q_{-d} \circ \Pi \big) \circ \exp(-u \Pi) 
$$
$$
= \exp(u \Pi) \circ  Q_{[~,~]_{\om}} \circ \exp(-u \Pi) =  Q_{[~,~]_{\om}}.  
$$

Hence both $\Psi_L$ and $\Psi_R$ satisfy the same formal ordinary differential 
equation 
$$
\frac{d}{d u} \Psi =  Q_{[~,~]_{\om}}
$$ 
with the same initial condition $\Psi \big|_{u = 0} = Q_{-d}$.

Therefore, $\Psi_L = \Psi_R$ and \eqref{exp-Pi-good} follows. 

The similar argument shows that the same operator $\exp(\Pi)$
defines an $L_{\infty}$ quasi-isomorphism 
$$
(L, -d + d_t, 0) \leadsto (L, -d + d_t, [~,~]_{\om}).
$$

Thus the corollary is proved. 
\end{proof}
~\\ 

\section{The proof of Theorem \ref{thm:main}}
\label{sec:proof}

Let us denote by $\ti{\cG} (\cL_{\Om}, -d, 0)$  the 
full subgroupoid of $\cG (\cL_{\Om}, -d, 0)$ 
whose set of objects is $\MC(\ti{\cL}_{\Om}, -d, 0)$.

Clearly, 
\begin{equation}
\label{TL0}
\pi_0 \big( \ti{\cG} (\cL_{\Om}, -d, 0) \big) 
~\cong~ \bigoplus_{q \geq 0} \frac{1}{\ve^{q-1}} \, \big( \mm \, \bs^q\, H^q(M, \bbC)[[\ve, \ve_1, \dots, \ve_g]] \big)^2,
\end{equation}
where $ \big( \mm \, \bs^q\, H^q(M, \bbC)[[\ve, \ve_1, \dots, \ve_g]] \big)^2$ 
is the subspace of degree $2$ elements in  
$$
\mm \, \bs^q\, H^q(M, \bbC)[[\ve, \ve_1, \dots, \ve_g]]\,. 
$$
 
Since the $L_{\infty}$ quasi-isomorphism \eqref{exp-Pi} is strictly invertible, it  
induces a bijection of sets 
\begin{equation}
\label{exp-Pi-on-MC}
\exp(\Pi)_*  \,:\,  \MC(\ti{\cL}_{\Om}, -d, 0) \, \stackrel{\cong}{\longrightarrow} \,  \MC(\ti{\cL}_{\Om}, -d,  [~,~]_{\om}). 
\end{equation}

Moreover the second part of Corollary \ref{cor:exp-Pi} implies that, if $\mu_1, \mu_2 \in  \MC(\ti{\cL}_{\Om}, -d, 0)$
are connected by a $1$-cell in $\mMC_{\bul} (\cL_{\Om}, -d, 0)$ then $\exp(\Pi)_*(\mu_1)$ and 
$\exp(\Pi)_*(\mu_2)$ are connected by a $1$-cell in $\mMC(\cL_{\Om}, -d,  [~,~]_{\om})$. 
 
In other words, \eqref{exp-Pi-on-MC} gives us a well defined map 
\begin{equation}
\label{Te-Pi}
\Te_{\Pi} : \pi_0 \big( \ti{\cG} (\cL_{\Om}, -d, 0) \big) \to \TL_{\Om}\,.
\end{equation}

Let us prove that 
\begin{claim}
\label{cl:Te-Pi}
The map $\Te_{\Pi}$ is a bijection of sets. 
\end{claim}
\begin{proof}
The surjectivity of $\Te_{\Pi}$ follows from the fact that the map
\eqref{exp-Pi-on-MC} is a bijection. 

To prove the injectivity, we consider two MC elements $\mu_1, \mu_2 \in \MC(\ti{\cL}_{\Om}, -d, 0)$
and assume that the MC elements  $\exp(\Pi)_*(\mu_1)$ and $\exp(\Pi)_*(\mu_2)$ are connected by 
a $1$-cell in  $\mMC(\cL_{\Om}, -d,  [~,~]_{\om})$. Due to the second part of Corollary \ref{cor:exp-Pi}, 
the MC elements 
$$
\exp(-\Pi)_* \circ \exp(\Pi)_*(\mu_1) = \mu_1
\qquad \textrm{and} \qquad 
\exp(-\Pi)_* \circ \exp(\Pi)_*(\mu_2) = \mu_2
$$
are connected by a $1$-cell in $\mMC(\cL_{\Om}, -d, 0)$. 

Thus  $\Te_{\Pi}$ is indeed injective. 
\end{proof}

By putting all the things together, we can now complete the proof of Theorem \ref{thm:main}.

Indeed, due to Proposition \ref{prop:MC-in-quest}, we have a bijection 
$$
\TL \cong \TL^{\tw}.
$$
The map \eqref{TL-PV-wt-TL} induced by the $L_{\infty}$ quasi-isomorphism \eqref{cU-al-needed}
gives us a bijection
$$
\TL_{\sPV} \cong  \TL^{\tw}. 
$$

The isomorphism $\sJ_{\om}$ from \eqref{ticL-PV-ticL-Om} gives us a bijection 
$$
\TL_{\sPV} \cong  \TL_{\Om}
$$
and, finally, the map $\exp(\Pi)_*$ induces a bijection 
$$
\pi_0 \big( \ti{\cG} (\cL_{\Om}, -d, 0) \big)  \cong  \TL_{\Om}.
$$

Since 
$$
\pi_0 \big( \ti{\cG} (\cL_{\Om}, -d, 0) \big) ~\cong~
\bigoplus_{q \geq 0} \frac{1}{\ve^{q-1}} \, \big( \mm \, \bs^q\, H^q(M, \bbC)[[\ve, \ve_1, \dots, \ve_g]] \big)^2\,,
$$
the proof of Theorem \ref{thm:main} complete. \qed

\section{Final remarks}
\label{sec:conj}

In this concluding section, we would like to pose two open questions. 

It is known \cite{stable}, \cite[4.6.3]{K} and \cite{K-conj} that there are infinitely many homotopy 
classes of formality quasi-isomorphisms from the dg Lie algebra of polyvector fields to the dg Lie 
algebra of polydifferential operators. So it would be interesting to determine whether 
the bijection from Theorem \ref{thm:main} depends on the homotopy type of the 
formality quasi-isomorphism for polydifferential operators. 

We believe that 
\begin{conj}
\label{conj:Fedosov}
There is a construction of a bijection between $\TL$ and the set of formal series in \eqref{veveve-H} 
which bypasses the use of Kontsevich's formality theorem. 
This construction comes from an appropriate generalization of the zig-zag of quasi-isomorphisms of 
dg Lie algebras from paper \cite{Hoch-dR}. 
\end{conj}
\begin{conj}
\label{conj:F-to-K}
The bijection between $\TL$ and the set of formal series in \eqref{veveve-H} coming from the above conjectural construction 
coincides with the bijection produced in this paper for any choice of a formality 
quasi-isomorphism \eqref{cU-fixed}.
\end{conj}
If true, the statement of Conjecture \ref{conj:F-to-K} would imply that the constructed bijection 
between $\TL$ and  the set of formal series in \eqref{veveve-H} does not depend on the 
choice of a formality quasi-isomorphism \eqref{cU-fixed}.  

We believe that a solution of Conjecture \ref{conj:Fedosov} will allow us to produce 
explicit examples of $A_{\infty}$-structures on $\cO(M)[[\ve, \ve_1, \dots, \ve_g]]$ 
corresponding to formal series in  \eqref{veveve-H} for a large class of symplectic manifolds. 
To our genuine surprise, Kontsevich's 
quasi-isomorphism \cite{K} is not very helpful for computing these $A_{\infty}$-structures 
even in the case when $M$ is an even dimensional torus with the standard symplectic structure!

We also believe that Conjecture \ref{conj:F-to-K} can be tackled using the ideas from \cite{BDW}.

\appendix

\section{Properties of the coderivation $\Pi$}
\label{app:Pi}
In this appendix, we prove the properties of the coderivation $\Pi$ (see \eqref{Pi}) of 
the coalgebra \eqref{S-ti-cL} listed in Proposition \ref{prop:Pi}. 

It is straightforward to see that $\Pi$ has degree zero. 

Since
$$
p \circ [\Pi , Q_{-d}] \,( \bs^{-2} \, \eta_1 \, \dots \, \bs^{-2} \, \eta_n ) = 
p \circ Q_{[~,~]_{\om}} \, ( \bs^{-2} \, \eta_1 \, \dots \, \bs^{-2} \, \eta_n ) = 0
$$
if $n \neq 2$, to prove \eqref{Pi-nice}, it suffices to show that 
\begin{equation}
\label{Pi-goal}
p \circ (\Pi \circ Q_{-d} - Q_{-d} \circ \Pi)\, (\bs^{-2} \, \la\, \bs^{-2} \, \eta) = 
p \circ Q_{[~,~]_{\om}} \,  (\bs^{-2} \, \la\, \bs^{-2} \, \eta),
\end{equation}
where $\la$ and $\eta$ are homogeneous vectors in
$$
\frac{1}{\ve^{k-1}} \,\mm \, \Om^k(M) [[\ve, \ve_1, \dots, \ve_g]]
\qquad
\textrm{and}
\qquad
 \frac{1}{\ve^{r-1}} \,\mm \, \Om^r(M) [[\ve, \ve_1, \dots, \ve_g]]
$$
respectively. 

Let $U \subset M$ be an open coordinate subset with coordinates $x^1, x^2, \dots, x^m$ and 
$$
\la \big|_{U}  =  dx^{i_1} dx^{i_2} \dots dx^{i_k} \, \la_{i_1 \dots i_k}\,, \qquad
\eta  \big|_{U} = dx^{j_1} dx^{j_2} \dots dx^{j_r} \, \eta_{j_1 \dots j_r}\,,
$$
where $\la_{i_1 \dots i_k} \in C^{\infty}(U)[[\ve, \ve_1, \dots, \ve_g]][\ve^{-1}]$, 
$\eta_{j_1 \dots j_r} \in C^{\infty}(U)[[\ve, \ve_1, \dots, \ve_g]][\ve^{-1}]$ and summation 
over repeated indices is assumed. 

The direct computation shows that 
$$
p \circ \Pi \circ Q_{-d} \, (\bs^{-2} \, \la\, \bs^{-2} \, \eta) \big|_{U} = 
$$
\begin{equation}
\label{needed}
(-1)^{|\la|} \bs^{-2}\, \ve \al^{ij}(x)  
dx^{i_1} dx^{i_2} \dots dx^{i_k} \, \del_{x^i} \la_{i_1 \dots i_k} \,
\frac{\del \eta}{\del \, d x^j}  
\end{equation}
\begin{equation}
\label{needed1}
- \bs^{-2}\, \ve \al^{ij}(x) \frac{\del \la}{\del d x^i} \,
 dx^{j_1} dx^{j_2} \dots dx^{j_r} \, \del_{x^j} \eta_{j_1 \dots j_r}
\end{equation}
\begin{equation}
\label{unwanted}
- (-1)^{|\la|} \bs^{-2}\, \ve \, k\,  \al^{i_1 j}(x)
dx^{i} dx^{i_2} \dots dx^{i_k} \, \del_{x^i} \la_{i_1 \dots i_k}
\frac{\del \eta}{\del \, d x^j}  
\end{equation}
\begin{equation}
\label{unwanted1}
+ \bs^{-2}\, \ve \, r\,  \al^{i j_1}(x) \frac{\del \la}{\del d x^i} \,
 dx^{j} dx^{j_2} \dots dx^{j_r} \, \del_{x^j} \eta_{j_1 \dots j_r}\,.
\end{equation}

Furthermore, 
$$
- p \circ Q_{-d} \circ \Pi \, (\bs^{-2} \, \la\, \bs^{-2} \, \eta) \big|_{U} = 
$$
\begin{equation}
\label{needed11}
(-1)^{\la} \,\bs^{-2}\, \ve \, d x^t \del_{x^t} \al^{ij}(x) \, \frac{\del \la}{\del d x^i} \, \frac{\del \eta}{\del d x^j}
\end{equation}
\begin{equation}
\label{unwanted11}
(-1)^{|\la|} \bs^{-2}\, \ve \, k\,  \al^{i_1 j}(x)
dx^{i} dx^{i_2} \dots dx^{i_k} \, \del_{x^i} \la_{i_1 \dots i_k}
\frac{\del \eta}{\del \, d x^j}  
\end{equation}
\begin{equation}
\label{unwanted111}
- \bs^{-2}\, \ve \, r\,  \al^{i j_1}(x) \frac{\del \la}{\del d x^i} \,
 dx^{j} dx^{j_2} \dots dx^{j_r} \, \del_{x^j} \eta_{j_1 \dots j_r}\,.
\end{equation}

Term \eqref{unwanted11} (resp. term \eqref{unwanted111}) cancels term
\eqref{unwanted} (resp. term \eqref{unwanted1}) in the left hand side of \eqref{Pi-goal}. 
Moreover the sum of terms \eqref{needed}, \eqref{needed1} and \eqref{needed11} coincides with 
$$
p \circ Q_{[~,~]_{\om}} \,  (\bs^{-2} \, \la\, \bs^{-2} \, \eta) \big|_{U}\,.
$$

Thus identity \eqref{Pi-goal} (and hence \eqref{Pi-nice}) is proved. 

To prove \eqref{Pi-with-Q-bracket}, we observe that 
$$
p \circ \Pi \circ  Q_{[~,~]_{\om}} \,(  \bs^{-2} \, \eta_1 \, \dots \, \bs^{-2} \, \eta_n ) =   
p \circ Q_{[~,~]_{\om}} \circ \Pi \,(  \bs^{-2} \, \eta_1 \, \dots \, \bs^{-2} \, \eta_n ) = 0
$$
if $n \neq 3$. 

For $n=3$, a direct computation shows that 
$$
p \circ \Pi \circ  Q_{[~,~]_{\om}} \,(  \bs^{-2} \, \eta_1 \,  \bs^{-2} \, \eta_2 \, \bs^{-2} \, \eta_3 ) - 
p \circ Q_{[~,~]_{\om}} \circ \Pi \,(  \bs^{-2} \, \eta_1 \,  \bs^{-2} \, \eta_2 \, \bs^{-2} \, \eta_3 ) = 
$$
$$
-2 \Big(\, (-1)^{|\eta_2|} \, \bs^{-2}\, \ve^2 \,  \al^{i_1 j_1}(x) (\del_{x^{i_1}} \al^{ij}(x))\, 
\frac{\del \eta_1}{\del\, d x^{i}} \frac{\del \eta_2}{\del\, d x^{j}} \frac{\del \eta_3}{\del\, d x^{j_1}} 
$$
$$
+(-1)^{|\eta_3| + |\eta_1| (|\eta_2| + |\eta_3|)} \,
  \bs^{-2}\, \ve^2 \,  \al^{i_1 j_1}(x) (\del_{x^{i_1}} \al^{ij}(x))\, 
\frac{\del \eta_2}{\del\, d x^{i}} \frac{\del \eta_3}{\del\, d x^{j}} \frac{\del \eta_1}{\del\, d x^{j_1}} 
$$
$$
+(-1)^{|\eta_1| + |\eta_3| (|\eta_1| + |\eta_2|)} \,
  \bs^{-2}\, \ve^2 \,  \al^{i_1 j_1}(x) (\del_{x^{i_1}} \al^{ij}(x))\, 
\frac{\del \eta_3}{\del\, d x^{i}} \frac{\del \eta_1}{\del\, d x^{j}} \frac{\del \eta_2}{\del\, d x^{j_1}} \, \Big)  =
$$
$$
2 (-1)^{|\eta_2|+1} \, \bs^{-2}\, \ve^2 \,  
\big(
\al^{i_1 j_1}(x) \del_{x^{i_1}} \al^{ij}(x) +
\al^{i_1 i}(x) \del_{x^{i_1}} \al^{j j_1}(x) +
\al^{i_1 j}(x) \del_{x^{i_1}} \al^{j_1 i}(x) 
\big)\,
\frac{\del \eta_1}{\del\, d x^{i}} \frac{\del \eta_2}{\del\, d x^{j}} \frac{\del \eta_3}{\del\, d x^{j_1}} \,.
$$
Thus \eqref{Pi-with-Q-bracket} is a consequence of the Jacoby identity for the Poisson structure $\al$.
Proposition \ref{prop:Pi} is proved.  \qed

~\\

\noindent\textsc{Department of Mathematics,\\
Temple University, \\
Wachman Hall Rm. 638\\
1805 N. Broad St.,\\
Philadelphia PA, 19122 USA \\
\emph{E-mail addresses:} {\bf elif238@gmail.com}, {\bf vald@temple.edu}}

\end{document}